\documentclass[a4paper]{amsart}
\input{article.sty}

\title[CH manifolds of very negative curvature]{Some functional properties on Cartan-Hadamard manifolds of very negative curvature}
\date{\today}

\author{Ludovico Marini*} 
\address[L. Marini]{Dipartimento di Matematica e Applicazioni,
Università degli Studi di Milano-Bicocca, Via R. Cozzi 55, I-20125, Milano}
\email[L. Marini - corresponding author]{l.marini9@campus.unimib.it}

\author{Giona Veronelli}
\address[G. Veronelli]{Dipartimento di Matematica ed Applicazioni, Università degli Studi di Milano-Bicocca, Via R. Cozzi 55, I-20125, Milano}
\email[G. Veronelli]{giona.veronelli@unimib.it}

\begin{document}
\begin{abstract}
In this paper we consider Cartan-Hadamard manifolds (i.e. simply connected of non-positive sectional curvature) whose negative Ricci curvature grows polynomially at infinity. 
We show that a number of functional properties, which typically hold when the curvature is bounded, remain true in this setting. 
These include the characterization of Sobolev spaces on manifolds, the so-called Cald\'eron-Zygmund inequalities and the $L^p$-positivity preserving property, i.e. $u\in L^p\ \&\ (-\Delta + 1)u\ge 0 \Rightarrow u\ge 0$. 
The main tool is a new class of first and second order Hardy-type inequalities on Cartan-Hadamard manifolds with a polynomial upper bound on the curvature.

In the last part of the manuscript we prove the $L^p$-positivity preserving property, $p\in[1,+\infty]$, on manifolds with subquadratic negative part of the Ricci curvature. This generalizes an idea of B. G\"uneysu and gives a new proof of a well-known condition for the stochastic completeness due to P. Hsu.
 \end{abstract}
\maketitle

\section{Introduction}
\label{sec:introduction}

A major task for geometric analysts consists in determining under which assumptions, and to what extent, certain properties typical of the Euclidean space have their counterparts on a given complete, non-compact Riemannian manifold. 
The properties one is interested in include for instance certain functional inequalities, the behavior of solutions of PDEs, the characterization of some functional spaces, spectral properties, and so on. 
A common set of assumptions which ensure that the manifold $M$ at hand is in a sense ``similar'' to the the Euclidean space (locally, but uniformly) is a constant lower bound on the Ricci curvature, or $|\Ric| \in L^\infty$ together with a positive lower bound on the injectivity radius. 
In this spirit, we consider the following problems.

\subsection*{A}
On a Riemannian manifold $(M, g)$ one disposes of several, a priori different, definitions for the Sobolev space of order $k \in \mathbb{N}$ and integrability class $p \in [1,+\infty]$. 
For instance, one can define $W^{k,p}(M)$ as the space of $L^p$-functions whose covariant (distributional) derivatives are in $L^p$ up to the order $k$: 
\begin{equation}
    \label{eq:W^{k,p}}
    W^{k, p}(M) \coloneqq \{ f \in L^p(M) : \nabla^j f \in L^p(M), \quad j = 0, \ldots k\}. 
\end{equation}
This turns out to be a Banach space once endowed with the usual norm
\begin{equation*}
    \Vert u \Vert_{W^{k, p}} \coloneqq \sum_{j=0}^k \Vert \nabla^j u \Vert_{L^p}. 
\end{equation*}
Thanks to a generalized Meyers-Serrin-type theorem, \cite{GGP2017}, if $p \in [1, + \infty)$ this space can be characterized as the closure of $W^{k, p}(M) \cap C^\infty(M)$ with respect to $\Vert \cdot\Vert_{W^{k, p}}$, which is quite useful in applications. 
Alternatively, one can define the space $W_0^{k,p}(M)$ as the closure of compactly supported smooth functions $C^\infty_0(M)$ with respect to the Sobolev norm $\Vert \cdot\Vert_{W^{k, p}}$, 
\begin{equation}
    \label{eq:W^{k,p}_0}
    W_0^{k,p} \coloneqq \overline{C^\infty_0(M)}^{\Vert \cdot\Vert_{W^{k, p}}}.
\end{equation}
Finally, for even orders one can consider $H^{2m, p}(M)$ as the space of $L^p$ functions whose iterations of the (distributional) Laplace-Beltrami operator are in $L^p$ up to order $m$, i.e., 
\begin{equation}
    H^{2m, p}(M) \coloneqq \{f \in L^p(M): \Delta^j f \in L^p(M), \quad j = 0, \ldots m\}, 
\end{equation}
endowed with the norm:
\begin{equation*}
     \Vert u \Vert_{H^{2m, p}} \coloneqq \sum_{j=0}^m \Vert \Delta^j u \Vert_{L^p}. 
\end{equation*}
In the Euclidean setting, $M = \R^n$, and on closed manifolds, the three spaces coincide. 
On an arbitrary Riemannian manifold one always has $W^{1,p}(M) = W^{1,p}_0(M)$, \cite{A1976}, whereas $k=2$ is the first non-trivial order where in general one can only conclude that
\begin{equation*}
    W^{2,p}_0(M) \subseteq W^{2, p}(M) \subseteq H^{2,p}(M). 
\end{equation*}
Nonetheless, if $|\Ric| \in L^\infty$ and the injectivity radius does not vanishes, it is actually possible to prove that $ W^{2,p}_0(M) = W^{2, p}(M) = H^{2,p}(M)$, see \cite{GP2015, He99}. 
See also \cite{V2020} for a detailed introduction to the problem. 
Both proofs rely on a computation in a harmonic coordinate system which, together with a covering argument, allows to reduce the Riemannian problem to the Euclidean setting. 
It is worth noticing that the result is also true for higher order $k$ if we require also that $|\nabla^j \Ric| \in L^\infty$ for $j = 0, \ldots, k-2$. 
In the Hilbert case ($p = 2$),  where a Bochner formula is available, a lower bound on Ricci curvature is actually enough, \cite{Ba14}. 

\subsection*{B}
The second problem we consider is the existence of $W^{2,p}$ regularity estimates for the solutions of the Poisson equation on a Riemannian manifold $(M, g)$; see \cite{P2020} for a nice recent survey on the topic. 
More specifically, we are interested in a-priori $L^p$-Hessian estimates of the form
\begin{equation}
\label{eq:CZintro}
\Vert \nabla^2\varphi \Vert_{L^p} \leq C \left[ \Vert \Delta \varphi \Vert_{L^p} + \Vert \varphi \Vert_{L^p} \right] \qquad \forall \varphi \in C^\infty_0(M)
\end{equation}
where $C>0$ is a positive constant.
Here $p \in (1, + \infty)$ and $\nabla^2 \varphi$ denotes the Hessian of $\varphi$, i.e., the second order covariant derivative. 
Such inequalities, known in literature as $L^p$-Calder\'on-Zygmund ($CZ(p)$) inequalities, were first established in a work by A. Calder\'on and A. Zygmund, \cite{CZ1952}, in the Euclidean setting, where in fact one has the stronger
\begin{equation*}
    \Vert \nabla^2 \varphi \Vert_{L^p} \leq C \Vert \Delta \varphi \Vert_{L^p} \qquad \forall \varphi \in C^\infty_0(\R^n). 
\end{equation*}
Note that the limit cases $CZ(1)$ and $CZ(+\infty)$ have been left out as they fail even in the Euclidean space, \cite{O1962,dLM1962}. 
It turn out that the validity of a $CZ(p)$ inequality on a Riemannian manifold implies the equality of the three Sobolev spaces defined in (A), for details we refer to \cite[Remark 2.1]{V2020} or \Cref{rmk:CZ(p) implies equality of sobolev} below.
As a matter of fact, one can ensure the validity of \eqref{eq:CZintro} under the same assumptions of $|\Ric|\in L^\infty$ and non-vanishing injectivity radius, \cite[Theorem C]{GP2015}. 
Furthermore, if $p = 2$ a lower bound on Ricci curvature is enough, \cite[Theorem B]{GP2015}. 

\subsection*{C}
Finally, we consider a positivity property for the solutions of $-\Delta u + u \geq 0$ on a complete Riemannian manifold. 
Note that in this paper $-\Delta$ has non-negative spectrum.
\begin{definition}
\label{def:Lp ppp}
A complete Riemannian manifold $(M, g)$ is said to be \textit{$L^p$-positivity preserving},  $p \in [1, +\infty]$, if the following implication holds true for every $u \in L^p(M)$
\begin{equation}
    \label{eq:Lp pp}
    (-\Delta + 1)u \geq 0 \text{ as a distribution }\Rightarrow u \geq 0.
\end{equation}
\end{definition}
Recall that $(-\Delta + 1)u \geq 0$ in the sense of distributions if the following inequality holds
\begin{equation*}
    \int_M u(-\Delta + 1)\phi dV_g \geq 0 \qquad \forall \phi \in C^\infty_0(M), \phi \geq 0.
\end{equation*}
This definition was introduced by B. G\"uneysu in \cite{Gu17}.
When $p = +\infty$, the $L^\infty$-positivity preserving property implies stochastic completeness while the $L^2$ case, yields the essential self-adjointness of the Schr\"odinger operator $-\Delta + V: C^\infty_0(M) \to L^2(M)$ for any non-negative $L^2_\loc$ potential $V$. 
This latter implication and the fact that $-\Delta + V$ is known to be essentially self-adjoint in $L^2(M)$ for $L^2_\loc$ non-negative potentials \cite{BMS02,GP2013}, lead M. Braverman, O. Milatovic and M. Shubin to propose the following conjecture, \cite[Conjecture P]{BMS02}.
\begin{conjecture}[BMS-conjecture]
If $(M, g)$ is geodesically complete, then $M$ is $L^2$-positivity preserving. 
\end{conjecture}
In the Euclidean case, the $L^2$-positivity preserving was proved by T. Kato using the fact that $-\Delta + 1 : \mathcal{S}'(\R^n) \to \mathcal{S}'(\R^n)$ induces an isomorphism on the space of tempered distributions whose inverse is positivity preserving, see \cite{K1972}. 
In the Riemannian setting, even though the BMS conjecture is still open in its full generality, one can prove that if $\Ric$ is bounded form below, then $M$ is $L^p$-positivity preserving on the whole scale $p \in [1, +\infty]$, \cite[Theorem XIV.31]{Gu17}.
For a complete introduction to the topic we refer to the survey \cite{Gu17BMS} as well as \cite[Section XIV.5]{Gu17}, \cite[Appendix B]{BMS02} or \cite{G2016}.

\vspace{\baselineskip}
\vspace{\baselineskip}

Some of the aforementioned results can be slightly improved by allowing a small explosion on the non-negative part of $\Ric$.
For instance, the equivalence of the Sobolev spaces 
${W^{2,p}(M)=W_0^{2,p}(M)}$, $p\in[1,\infty)$,  still holds if we allow $|\Ric| \leq br^2$ with a small decay of the injectivity radius, while to prove that ${W^{2,2}(M)=W_0^{2,2}(M)}$ and the $L^p$-positivity preserving property  $\Ric \geq -br^2$ is enough.
For reference on the first problem see \cite{IRV2019,IRV2020,HMRV20}, for the second one see \cite[p. 95, Section XIV.C]{Gu17} for $p\in [2,\infty)$ and \Cref{thm:Lp pp subquadratic} below for the whole range $p\in[1,+\infty]$. 
Nevertheless, the above results fail in general if we drop the curvature (and injectivity) assumptions and allow the bound on $|\Ric|$ to grow very fast at $\infty$.
Counterexamples with very unbounded curvature have been found in \cite{V2020,HMRV20} for (A) and in \cite{GP2015,Li,dePNZ,MV} for (B) while, to the best of our knowledge, the BMS conjecture remains open in its full generality. 
Note that the above counterexamples are characterized by an oscillatory behavior of the Ricci curvature which diverges in the negative and positive part in \cite{GP2015, Li, V2020} and in the positive part only in \cite{MV, HMRV20, dePNZ}. 

\vspace{\baselineskip}
\vspace{\baselineskip}

In this paper we show that several of the above properties still hold if one allows the curvature to become increasingly negative at infinity, possibly very fast, but in a controlled way. 
In particular, we consider a Cartan-Hadamard manifold $(M, g)$ (i.e. a simply-connected complete Riemannian manifold of non-positive sectional curvature) and assume that the Ricci curvature of $M$, $\Ric$, is controlled both from above and below polynomially at infinity. Namely,
\begin{equation}
    \label{bound intro}
    -b\, r^{\beta}(x) \leq \Ric(x) \leq - a\, r^\alpha(x), 
\end{equation}
outside a compact set, where $r(x)$ is the Riemannian distance of $x$ from a fixed reference point $o\in M$ and $a$ and $b$ are positive constants. 
Then, for suitable choices of the exponents $0\leq\alpha\leq\beta$ we are able to prove the following results. 

\begin{mytheorem}
\label{th_main}
Let $(M,g)$ be a Cartan-Hadamard manifold satisfying \eqref{bound intro} for some $a,b>0$. 
\begin{enumerate}[label=(\alph*)]
    \item If $\alpha \ge 0$ and $\beta=2\alpha +2$, then $W_0^{2,p}(M)=W^{2,p}(M)$ for every $p\in(1,+\infty)$.
    \item If $\alpha=\beta \ge 0$, then the $L^2$-Calder\'on-Zygmund inequality (i.e., \eqref{eq:CZintro} with $p=2$) holds on $M$.
    \item If $\alpha \ge 0$ and $\beta=\alpha +2$, then  $M$ is $L^p$-positivity preserving for all $p\in [2,+\infty)$. In particular, the BMS conjecture is satisfied for this class of manifolds. 
\end{enumerate}
\end{mytheorem}

\begin{remark}
Note that in (a) and (c) we only require the radial Ricci curvature to satisfy \eqref{bound intro} while in (b) we need \eqref{bound intro} to hold in the sense of quadratic forms. 
Naturally, if one substitutes $\Ric$ in \eqref{bound intro} with the sectional curvature, the results still hold and are actually somewhat easier to prove.
This is due to the fact that in the proof we use a Laplace comparison theorem for Ricci bounded from above which holds on Cartan-Hadamard manifolds; see Subsection \ref{ss:Comparison} and \Cref{rmk:uncommon bounds}. 
\end{remark}

\begin{remark}
An additional property that extends from the Euclidean setting to the case of Riemannian manifolds with Ricci curvature bounded form below and non-vanishing injectivity radius
is the validity of an $L^p$-Sobolev inequality of the form 
\begin{equation}
\label{SobIntro}
\Vert\varphi\Vert_{L^q(M)} \le C(\Vert\nabla\varphi\Vert_{L^p(M)}+\Vert\varphi\Vert_{L^p(M)}),\qquad \forall \varphi\in C^\infty_c(M),
\end{equation}
where $1\le p \le n$, $q=np/(n-p)$, and $C>0$. For reference see \cite[Theorem 3.2]{He99}.
Also in this case there are known counterexamples if one drops the curvature assumptions \cite[Proposition 3.4]{He99}.
On Cartan-Hadamard manifolds, however, the Sobolev inequality \eqref{SobIntro} is satisfied without further curvature assumptions as a consequence of the isoperimetric inequality, \cite{HS1974}.
\end{remark}

Unlike the case of manifolds with lower bounded Ricci curvature, it is impossible to obtain Theorem \ref{th_main} (c) for $p=\infty$.
Indeed, as observed by B. G\"uneysu in \cite{G2016}, the $L^\infty$-positivity preserving property implies the stochastic completeness of the manifold at hand. 
See also \Cref{rmk:stoch compl} below. 
It turns out that Cartan-Hadamard manifolds satisfying $\Ric \leq - ar^\alpha$ for $\alpha > 2$ are not stochastically complete; see Theorem \ref{thm:comparison stochastic completeness} below. 
Conversely, one can prove that
\begin{mytheorem}
\label{thm:LpPP intro}
Let $(M,g)$ be a complete Riemannian manifold satisfying
\begin{equation*}
    -\lambda^{2}(r(x)) \leq \Ric(x) \quad \forall x \in M\setminus B_{R_0},
\end{equation*}
with $\lambda$ given by
\begin{equation*}
    \lambda(t) = \alpha t \prod_{j = 0}^k \log^{[j]}(t)
\end{equation*}
where $\alpha>0$, $k\in\mathbb{N}$ and $\log^{[j]}(t)$ stands for the $j$-th iterated logarithm. 
Then $M$ is $L^p$-positivity preserving for any $p\in[1,\infty]$.
\end{mytheorem}
As a corollary of the $p=\infty$ case, we get in particular that a manifold at hand is stochastically complete. 
This gives a new proof of a celebrated condition for the stochastic completeness due to P. Hsu, \cite{Hsu1989}. See \Cref{rmk:hsu}.

Beyond their obvious topological triviality, the Cartan-Hadamard manifolds we consider in Theorem \ref{th_main} have also quite strong metrical properties. On the one hand, the lower bound $-br^\beta(x)$ for the Ricci curvature implies a Laplacian comparison, i.e, an upper control on $\Delta r$. 
This, in turn, permits to construct suitable Hessian and Laplacian cut-off functions. 
Namely, one gets the existence of a family of smooth cutoffs $\{\chi_R\}\in C^\infty_0(M)$ with $R>>1$ such that
\begin{enumerate}
    \item $\chi_R \equiv 1$ on $B_R$ and $\chi_R \equiv 0$ on $M \setminus \overline{B_{2 R}}$;
    \item $|\nabla \chi_R| \leq \frac{C_1}{R}$;
    \item $|\nabla^2 \chi_R| \leq C_2 R^{\frac{\beta}{2} -1}$, 
\end{enumerate}
with $C_1, C_2 > 0$ (see Lemma \ref{lem:existence of cutoffs}). Most of the strategies proposed in previous literature to approach the density problem or the $L^p$-positivity preservation are precisely based on the existence of suitable cut-off functions which have bounded covariant derivatives up to the second order, for instance in the subquadratic case.
Conversely, the control that we get on $|\nabla^2\chi_R|$ under our assumptions is not strong enough to allow us to obtain Theorem \ref{th_main} (a) and (c) by this strategy alone. 
The reason is essentially that, when $\beta >2$, the sole lower bound $\Ric\ge -br^\beta$ cannot guarantee that for any function $f$
\begin{equation}
\label{eq:condition on f}
f\in W^{2,p} \implies |\nabla^2\chi_R|f\in L^p.
\end{equation}
Instead, assuming also that $\Ric\le -a r^\alpha$, one gets \begin{equation}
\label{eq:condition on f weight}
f\in W^{2,p} \implies (r^\alpha f)\in L^p,
\end{equation}
see Theorem \ref{thm:rellich with weight sobolev}.
This latter relation, combined with the properties of the Hessian cut-off functions, yields \eqref{eq:condition on f}.

To obtain \eqref{eq:condition on f weight}, we exploit the validity on $\Omega\subset M$ of certain Hardy-type inequalities (obtained elaborating on ideas by L. D'Ambrosio and S. Dipierro, \cite{DD2014}) of the form
\begin{equation*}
   \int_\Omega \frac{|\nabla G |^p}{|G|^p} (-\log{G})^{\beta p} |f|^p dV_g \leq  \left(\frac{p}{p-1}\right)^p\int_\Omega (-\log{G})^{\beta p} |\nabla f |^p dV_g \qquad \forall f \in C^\infty_0(\Omega),
\end{equation*}
where $G \in C^\infty(\Omega)$ satisfies
\begin{enumerate}[label=(\roman*)]
    \item $-\Delta_p G \geq 0$ on $\Omega$;
    \item $0 \leq G \leq c < 1$;
\end{enumerate}
and $p \in (1, +\infty)$, see Theorem \ref{thm:log hardy} and \Cref{thm:rellich with weight sobolev}. 
Using a Laplacian comparison for Cartan-Hadamard manifolds, it turns out that an appropriate choice for $G$ is the Green function for the $p$-Laplacian of the model manifold $\widetilde{M}$ whose (radial) Ricci curvature is precisely $-a r^\alpha $. 

In order to prove Theorem \ref{th_main} (b) a further ingredient is needed. 
Using a special conformal deformation of $M$ based on the distance function, see \cite{IRV2019}, we prove first the validity of the disturbed infinitesimal Calder\'on-Zygmund inequality 
\begin{equation*}
    \Vert \nabla^2\varphi \Vert_{L^2} \leq A_1(\varepsilon) \left[ \Vert \Delta \varphi \Vert_{L^2} + \Vert \varphi \Vert_{L^2} \right] + A_2\varepsilon^2 \Vert r^\beta \varphi \Vert_{L^2} \qquad \forall \varphi \in C^\infty_0(M).
\end{equation*}
when $\Ric\ge - b\, r^\beta$; see Theorem 5.1. Then, one can conclude using again the Hardy-type inequalities. 

We conclude this introduction with some words on the novelty of our proof in Theorem \ref{thm:LpPP intro}. 
Following the strategy adopted in \cite[Appendix B]{BMS02} and \cite[Theorem XIV.31]{Gu17}, a key step in the proof of the $L^p$-positivity preserving is to show that for a given $\psi\in C^\infty_c$, there exists a positive solution $v\in C^\infty\cap W^{1,q}(M)$ of $-\Delta v + v =\psi$, with $1/q =1-1/p$. 
While standard elliptic regularity theory ensures that $v\in L^{q}(M)$ (and hence $\Delta v\in L^{q}(M)$), the fact that $\nabla v\in L^{q}(M)$ is non-trivial.
When $q\in (1,2]$ (i.e. $p\in [2,\infty)$) it is a consequence of the $L^q$-gradient estimates $\|\nabla v\|_{L^{q}}\le C(\| v\|_{L^{q}}+\|\Delta v\|_{L^q})$, \cite{CD2003}.
However, these estimates are not known a priori for $q>2$ when the negative part of $\Ric$ is unbounded.
Instead, we use a version of Li-Yau gradient estimates, \cite{BS2018}, to prove that $|\nabla v|(x)\le \lambda(r(x))v(x)$ outside a compact set. 
Hence, $\nabla v$ is ``almost'' in $L^q$, which is enough to our purpose.

\vspace{\baselineskip}
\vspace{\baselineskip}

The paper is organized as follows. 
In \Cref{sec:CH estimates} we construct and estimate the Green function for the $p$-Laplacian, $G_p$, on a model manifold with radial Ricci curvature $-a r^\alpha$. 
Then, using a Laplacian comparison for Cartan-Hadamard manifolds, we show that this function is $p$-superharmonic on a Cartan-Hadamard manifold whose Ricci curvature is bounded from above by $-a r^\alpha$. 
\Cref{sec:Hardy and Rellich} is devoted the proof of the Hardy-type inequalities whose weight is given in terms of $G_p$. 
In \Cref{sec:density}, \ref{sec:CZ(2)} and \ref{sec:L^p pp and BMS} we prove respectively part (a), (b) and (c) of \Cref{th_main}. 
In \Cref{sec:L^p pp and BMS} we also prove \Cref{thm:LpPP intro}. 

\begin{notation}
Throughout the paper, $C$ will denote a real positive constant whose value can change from line to line. 
Whenever appropriate, we will explicit its dependency on other constant or parameters.  
\end{notation}

\section{Estimates on Cartan-Hadamard manifolds}
\label{sec:CH estimates}
The goal of this section is to obtain asymptotic estimates for several geometric objects on Cartan-Hadamard manifolds whose (radial) Ricci curvature is bounded form above by $-ar^\alpha$.
\subsection{Model manifold case}
\label{subs:model manifold}
We begin by studying the geometry of model manifolds with prescribed Ricci curvature.
By direct computation we obtain asymptotic estimates for the $p$-Green function and the Laplacian of the Riemannian distance. 

Let $(\widetilde{M}, \widetilde{g}) = [0, +\infty) \times_{j} \mathbb{S}^{n-1}$ be a model manifold in the sense of E. R. Greene and H. Wu \cite{GW1979}, that is, $[0, +\infty) \times \mathbb{S}^{n-1}$ endowed with the metric
\begin{equation*}
    \widetilde{g} = dt^2 + j^2(t) d\theta^2,
\end{equation*}
where $d\theta^2$ is the standard metric on $\mathbb{S}^{n-1}$ and $j \in C^\infty((0, +\infty))$ such that $j > 0$ on $(0, +\infty)$, $j(0) = 0$, $j'(0) = 1$ and $j^{(2k)}(0) = 0$ for $k \in \mathbb{N}$. 
Denote with $\widetilde{\nabla}, \widetilde{\Delta}$ and $\widetilde{\Ric}$ the covariant derivative, Laplacian and Ricci tensor of $(\widetilde{M}, \widetilde{g})$ respectively, similarly $\widetilde{r}(x)$ is the Riemannian distance from the pole $o$ (so that $\widetilde{r}(t,\theta)=t$). 

Suppose $(\widetilde{M}, \widetilde{g})$ satisfies
\begin{equation*}
    \widetilde{\Ric}_o(x) = - (n-1)A^2 \widetilde{r}^\alpha(x),
\end{equation*}
where $A > 0$, $\alpha \geq 0$ and $\widetilde{\Ric_o}$ denotes Ricci curvature in the radial direction $\widetilde{\nabla}\widetilde{r}$. 
Since on model manifolds
\begin{equation*}
    \widetilde{\Ric}_o(x) = -(n-1) \frac{j''(\widetilde{r}(x))}{j(\widetilde{r}(x))},
\end{equation*}
$j$ needs to solve 
\begin{equation}
    \label{eq:ODE comparison}
    \begin{cases}
    j''(t) - A^2t^\alpha j(t) = 0 \\
    j(0) = 0, \quad j'(0) = 1
    \end{cases}
\end{equation}
for $t \in [0, + \infty)$.
By classical ODE theory we have  
\begin{equation}
    \label{eq:warping function model}
    j(t) = D \sqrt{t} I_\nu\left(2 A \nu t^{1 /2\nu}\right), \quad \nu = \frac{1}{\alpha + 2},
\end{equation}
where $D$ is a positive constant and $I_\nu(t)$ is the modified Bessel function of the first kind and order $\nu$, i.e., a positive solution of the Bessel equation
\begin{equation*}
    \label{eq:bessel}
    t^2I_\nu''(t) + tI_\nu'(t) -(t^2 + \nu^2)I_\nu(t) = 0, 
\end{equation*}
see for reference \cite{BMR2013} and \cite{L1972}.
Note that, since
\begin{equation*}
    I_\nu(t) \sim \frac{t^\nu}{\Gamma(\nu+1)} \qquad t \to 0
\end{equation*}
then
\begin{equation*}
    j(t) \sim D \frac{(2\nu A)^\nu}{\Gamma(\nu + 1)}t \qquad t \to 0, 
\end{equation*}
hence, $D = \frac{\Gamma(\nu + 1)}{(2\nu A)^\nu}$. 

Using the following asymptotic
\begin{equation}
    \label{eq:asymptotic bessel}
    I_\nu(t) \sim \frac{e^t}{\sqrt{2\pi t}} \quad t \to +\infty,
\end{equation}
we have
\begin{equation}
    \label{eq:asymptotic j}
    j(t) \sim D_0 t^{-\frac{\alpha}{4}} \exp\left( \frac{2A}{\alpha + 2} t^{1+\frac{\alpha}{2}}\right) \quad t \to + \infty, \quad D_0 = \frac{D}{\sqrt{4\nu\pi A}}.
\end{equation}
Since $\vol(\partial B_t) = \omega_n j(t)^{n-1}$ where $\omega_n$ is the volume of the Euclidean $n$-dimensional unit ball, \eqref{eq:asymptotic j} implies that 
\begin{equation*}
    \left(\frac{1}{\vol(\partial B_t)}\right)^{\frac{1}{p-1}} \in L^1(+\infty), \qquad \forall p>1. 
\end{equation*}
By \cite[Corollary 5.2]{T1999} we deduce that $(\widetilde{M}, \widetilde{g})$ is $p$-hyperbolic, that is, there exists a symmetric positive Green function for the $p$-Laplacian. 
Specifically, the positive $p$-Green function with pole $o \in \widetilde{M}$ is a radial function given by
\begin{equation}
    \label{eq:p green function}
    G_p(x) = G_p(t) \coloneqq \int_t^{+\infty}\left(\frac{1}{j(s)}\right)^{\frac{n-1}{p-1}} ds, \quad x = (t, \theta) \in (0, +\infty) \times \mathbb{S}^{n-1}.
\end{equation}
Using \eqref{eq:asymptotic j} we obtain 
\begin{equation}
    \label{eq:asymptotic G'}
    \partial_t G_p(t) \sim -D_1 t^{\frac{\alpha}{4}\frac{n-1}{p-1}} \exp\left(- A\frac{2}{\alpha + 2} \frac{n-1}{p-1} t^{1+\frac{\alpha}{2}}\right) \quad t \to +\infty,
\end{equation}
and
\begin{equation}
    \label{eq:asymptotic G}
    G_p(t) \sim D_2 t^{\frac{\alpha}{4}\left(\frac{n-1}{p-1}-2\right)} \exp\left(- A\frac{2}{\alpha + 2} \frac{n-1}{p-1} t^{1+\frac{\alpha}{2}}\right) \quad t \to +\infty, 
\end{equation}
where  $D_1, D_2$ are positive constants depending on $D, \alpha, n$ and $p$.
Note that $\partial_t G_p(t) < 0$ for all $t>0$.

Next, we compute the Laplacian of the Riemannian distance given by
\begin{equation*}
    \widetilde{\Delta}\widetilde{r} = (n-1)\frac{j'(\widetilde{r})}{j(\widetilde{r})}. 
\end{equation*}
By a simple computation we have
\begin{equation*}
    \frac{j'(t)}{j(t)} = \frac{1}{2t} + At^{\frac{\alpha}{2}}\frac{I_\nu'\left(2 A \nu t^{1 /2\nu}\right)}{I_\nu\left(2 A \nu t^{1 /2\nu}\right)}. 
\end{equation*}
Using the recurrence relation $2I_\nu'(t) =I_{\nu+1}(t) + I_{\nu-1}(t)$ and \eqref{eq:asymptotic bessel}, we conclude that $I_\nu'\sim I_\nu$ therefore
\begin{equation}
    \label{eq:asymptotic j'/j}
    \frac{j'(t)}{j(t)} \sim At^{\frac{\alpha}{2}} \qquad t \to + \infty. 
\end{equation}
Finally, using \eqref{eq:asymptotic j} once again we deduce
\begin{equation}
    \label{eq:int j}
    \int_0^t j^{n-1}(s) ds \sim D_3 t^{-\frac{\alpha}{4}(n+1)} \exp\left( \frac{2A}{\alpha + 2} (n-1) t^{1+\frac{\alpha}{2}}\right)
\end{equation}
for some positive constant $D_3$, so that 
\begin{equation}
    \label{eq:asymptotic stoch compl}
    \frac{\int_0^t j^{n-1}(s) ds}{j^{(n-1)}(t)} \sim D_4 t^{-\frac{\alpha}{2}}. 
\end{equation}

\subsection{Comparison results for Cartan-Hadamard manifolds}\label{ss:Comparison}
Next, we relate via the Laplacian comparison the above estimates to a Cartan-Hadamard manifold with a suitable bound on the Ricci curvature. 

Let $(M, g)$ be a Cartan-Hadamard manifold of dimension $n\geq2$ with a fixed pole $o \in M$ and suppose that
\begin{equation}
    \label{eq:upper bound ric}
    \Ric_o(x) \le -2(n-1)^2A^2 r^\alpha(x) \quad \forall x \in M\setminus B_{R_0}
\end{equation}
for some $A, R_0 > 0$ and $\alpha \geq 0$, here $r(x)$ denotes the Riemannian distance from the pole.
Let $(\widehat{M}, \widehat{g})$  be the model manifold of radial Ricci curvature
\begin{equation*}
    \widehat{\Ric}_o(\hat{x}) = -2(n-1)^3A^2 \widehat{r}^\alpha(\hat{x}), 
\end{equation*}
that is, $(\widehat{M}, \widehat{g}) = [0, +\infty) \times_{\widehat{j}} \mathbb{S}^{n-1}$ where
\begin{equation*}
    \widehat{j}(t) = \widehat{D} \sqrt{t} I_\nu\left(2 \widehat{A} \nu t^{1 /2\nu}\right), \quad \nu = \frac{1}{\alpha + 2}, \qquad \widehat{A} = \sqrt{2}(n-1)A. 
\end{equation*}
Since
\begin{equation*}
    \Ric_o(x) \le \frac{1}{n-1} \widehat{\Ric}_o(\hat{x}), 
\end{equation*}
for all $x \in M\setminus B_{R_0}$ and $\hat{x} \in \widehat{M}$ with $r(x) = \hat{r}(\hat{x})$, by \cite[Theorem 2.15]{X1996} and estimate \eqref{eq:asymptotic j'/j} we have 
\begin{equation*}
    \Delta r \geq \frac{\widehat{j}'(r)}{\widehat{j}(r)} \sim \widehat{A}r^{\frac{\alpha}{2}} =  \sqrt{2}(n-1)Ar^{\frac{\alpha}{2}} \qquad r \to + \infty. 
\end{equation*}
It follows that 
\begin{equation*}
    \frac{\widehat{j}'(r)}{\widehat{j}(r)} \sim \sqrt{2}(n-1)\frac{j'(r)}{j(r)}
\end{equation*}
where $j$ is as in \eqref{eq:warping function model}.
In particular, if $r>>1$ is large enough we can assume that
\begin{equation*}
   \Delta r \geq \frac{\widehat{j}'(r)}{\widehat{j}(r)} \geq (n-1)\frac{j'(r)}{j(r)}.
\end{equation*}
Note here that
\begin{equation*}
    \widetilde{\Delta} \widetilde{r} = (n-1)\frac{j'(\widetilde{r})}{j(\widetilde{r})}
\end{equation*}
is the Laplacian of the Riemannian distance on the model $(\widetilde{M}, \widetilde{g}) = [0, +\infty) \times_{j} \mathbb{S}^{n-1}$ considered in \Cref{subs:model manifold}.
In summary, we have the following comparison result. 

\begin{proposition}
\label{prop:comparison CH}
Let $(M, g)$ be a Cartan-Hadamard manifold of dimension $n\geq2$ with
\begin{equation*}
    \Ric_o(x) \le -2(n-1)^2A^2 r^\alpha(x) \quad \forall x \in M\setminus B_{R_0}
\end{equation*}
for some $A, R_0 > 0, \alpha \geq 0$.
Let $j$ be as in \eqref{eq:warping function model}, i.e., $j$ is the warping function of $(\widetilde{M}, \widetilde{g}) = [0, +\infty) \times_{j} \mathbb{S}^{n-1}$, model manifold of radial Ricci curvature
\begin{equation*}
    \widetilde{\Ric}_o(\widetilde{x}) = - (n-1)A^2 \widetilde{r}^\alpha(\widetilde{x}).
\end{equation*}
Then, if $r(x)>>1$ we have
\begin{equation}
    \Delta r \geq (n-1)\frac{j'(r)}{j(r)}.  
\end{equation}
\end{proposition}

\begin{remark}
\label{rmk:uncommon bounds}
The slightly uncommon bound we require in \eqref{eq:upper bound ric} is due to the fact that we make use of a Laplacian comparison result for Cartan-Hadamard manifolds which is different from the classical one and holds with an upper bound for the Ricci curvature instead of an upper bound for the sectional curvature.  
Note also that the constant $2$ in \eqref{eq:upper bound ric} is quite arbitrary: one could replace it with any constant strictly greater than $1$. 
\end{remark}

We begin with the following lemma. 

\begin{lemma}
\label{lem:comparison p-harmonic functions}
Let $(M, g)$ be a Cartan-Hadamard manifold and suppose that
\begin{equation}
    \Delta r \geq \phi(r) \text{ on }\Omega\subseteq M,
\end{equation}
for some $\phi \in C^0((0, +\infty))$ and $\Omega$ open. 
Let $v \in C^2(\mathbb{R})$ nonnegative and define $u(x) = v(r(x))$ for $x \in \Omega$. 
If $v' < 0$, then for all $p>1$ we have
\begin{equation}
    \label{eq:comparison p-Laplacian}
    \Delta_p u \leq |v'|^{p-2}(v'\phi(r) + (p-1)v''),
\end{equation}
on $\Omega \setminus \{o\}$.

\begin{proof}
Since $(M, g)$ is Cartan-Hadamard, then $r \in C^\infty(M \setminus \{o\})$ so that $u \in C^2(M \setminus \{o\})$. 
Suppose $v' < 0$, then 
\begin{align*}
\Delta_p u &= \Div(|\nabla u|^{p-2}\nabla u) = \Div(|v'|^{p-2}v' \nabla r) = |v'|^{p-2}(v'\Delta r + (p-1)v'') \\
&\leq |v'|^{p-2}(v'\phi(r) + (p-1)v''),
\end{align*}
on $\Omega\setminus \{o\}$. 
\end{proof}
\end{lemma}

\begin{remark}
Although it is not relevant to our work, we observe that if $v'>0$, then \eqref{eq:comparison p-Laplacian} holds with the opposite sign. 
\end{remark}

Combining \Cref{lem:comparison p-harmonic functions} with \Cref{prop:comparison CH} we obtain a comparison result for radial $p$-harmonic functions. 

\begin{proposition}
\label{prop:comparison p harmonic CH}
Let $(M, g)$ be a Cartan-Hadamard manifold satisfying \eqref{eq:upper bound ric} and let $(\widetilde{M}, \widetilde{g})$ be the model manifold as in \Cref{subs:model manifold}. 
Let $v \in C^2(\R)$ non-negative with $v'<0$ and define $u(x) = v(r(x))$ and $\widetilde{u}(\widetilde{x}) = v(\widetilde{r}(\widetilde{x}))$. 
Then $\Delta_p u(x) \leq \widetilde{\Delta}_p \widetilde{u}(\widetilde{x})$ for all $x \in M$ and $\widetilde{x} \in \widetilde{M}$ such that  $r(x) = \widetilde{r}(\widetilde{x}) >> 1$.
\begin{proof}
By \Cref{prop:comparison CH}, if $r(x)>>1$, then $\Delta r \geq (n-1) j'(r)/j(r)$, hence
\begin{equation*}
    \Delta_p u(x) \leq |v'(r(x))|^{p-2}\left[v'(r(x))(m-1)\frac{j'(r(x))}{j(r(x)} + (p-1)v''(r(x))\right] = \widetilde{\Delta}_p \widetilde{u}(\widetilde{x}). 
\end{equation*}
\end{proof}
\end{proposition}

In particular if we take $v(t) = G_p(t)$ as in \eqref{eq:p green function}, since $G_p$ defines the $p$-Green function on $(\widetilde{M}, \widetilde{g})$ we conclude that $G_p(x)$ is $p$-superharmonic on $(M, g)$ provided that $r(x) >> 1$.

\section{Hardy inequalities via Green function estimates}
\label{sec:Hardy and Rellich}

We now turn to the study of a class of functional inequalities on Riemannian manifolds which go under the name of Hardy or Rellich-type inequalities.
These inequalities have an interest of their own and are extensively studied in literature, especially in the case of Cartan-Hadamard manifolds. 
See \cite{BGGP2020, DD2014, DP2016, FLLM2021, KO2009, N2020, YSY2014} among others. 
With the help of a result by L. D'Ambrosio and S. Dipierro, \cite{DD2014}, we establish a new Hardy-type inequality on complete Riemannian manifolds possessing a non negative $p$-superharmonic function $G$.

\begin{theorem}
\label{thm:log hardy}
Let $(M,g)$ be a complete Riemannian manifold and $\Omega \subseteq M$ open. 
Fix $p > 1$ and let $G \in C^\infty(\Omega)$ such that
\begin{enumerate}[label=(\roman*)]
    \item $-\Delta_p G \geq 0$ on $\Omega$;
    \item $0 \leq G \leq c < 1$.
\end{enumerate}
Then, for any $\beta \geq 0$,
\begin{equation}
    \label{eq:log hardy}
   \int_\Omega \frac{|\nabla G |^p}{|G|^p} (-\log{G})^{\beta p} |f|^p dV_g \leq  \left(\frac{p}{p-1}\right)^p\int_\Omega (-\log{G})^{\beta p} |\nabla f |^p dV_g \qquad \forall f \in C^\infty_0(\Omega).
\end{equation}
\begin{proof}
Let $\delta > 0$ such that $G_\delta \coloneqq G + \delta < 1$ and define 
\begin{equation*}
    h \coloneqq -\frac{|\nabla G_\delta |^{p-2} \nabla G_\delta}{G_\delta^{p-1}}(-\log{G_\delta})^{\beta p}, \qquad A_h \coloneqq (p-1) \frac{|\nabla G_\delta |^{p}}{G_\delta^{p}}(-\log{G_\delta})^{\beta p}.
\end{equation*}
Since $G \in C^\infty(\Omega)$ and $G_\delta \geq \delta$ we have $|h|, A_h \in L^1_{\operatorname{loc}}(\Omega)$, furthermore, 
\begin{equation*}
     \frac{|h|^p}{A_h^{p-1}} = (p-1)^{1-p}(-\log{G_\delta})^{\beta p} \in L^1_{\operatorname{loc}}(\Omega).
\end{equation*}
Next, we estimate
\begin{align*}
    \Div(h) &= - \frac{(-\log{G_\delta})^{\beta p}}{G_\delta^{p-1}} \Delta_p G_\delta + (p-1)\frac{|\nabla G_\delta |^{p}}{G_\delta^{p}}(-\log{G_\delta})^{\beta p} + \beta p \frac{|\nabla G_\delta |^{p}}{G_\delta^{p}}(-\log{G_\delta})^{\beta p -1}\\
    &\ge (p-1)\frac{|\nabla G_\delta |^{p}}{G_\delta^{p}}(-\log{G_\delta})^{\beta p} = A_h. 
\end{align*}
Thanks to \cite[Lemma 2.10]{DD2014} we have 
\begin{equation*}
     \int_\Omega \frac{|\nabla G_\delta |^p}{|G_\delta|^p} (-\log{G_\delta})^{\beta p} |f|^p dV_g \leq \left(\frac{p}{p-1}\right)^p\int_\Omega (-\log{G_\delta})^{\beta p} |\nabla f |^p dV_g \qquad \forall f \in C^\infty_0(\Omega).
\end{equation*}
Since $-\log{G_\delta} \leq -\log{G}$ and $\nabla G_\delta = \nabla G$, letting $\delta \to 0$ and using Fatou's lemma yields \eqref{eq:log hardy}. 
\end{proof}
\end{theorem}

\begin{remark}
It is worth noticing that the results of \Cref{thm:log hardy} still hold even under more relaxed regularity assumptions. 
Notably, it suffices to have $G \in W^{1, p}_{\operatorname{loc}}(\Omega)$ and $-\Delta_p G \geq 0$ weakly on $\Omega$ to have the validity of \eqref{eq:log hardy}. 
Assumption \textit{(ii)} still needs to hold although it is always satisfied in applications.  
\end{remark}

Once we have the quite general \eqref{eq:log hardy}, we return to our setting, that is, $(M,g)$ is a Cartan-Hadamard manifold satisfying the Ricci upper bound  \eqref{eq:upper bound ric}. 
Under such curvature assumptions one easily gets that $(M, g)$ is a  $p$-hyperbolic manifold, i.e., there exists a symmetric positive Green kernel for the $p$-Laplacian.
Namely, if $\mathcal{G}_p(x)$ is the $p$-Green function with pole $o \in M$, it satisfies $\Delta_p \mathcal{G}_p(x) = 0$ for all $x \neq o$ and, thus, can be used as weight in \Cref{thm:log hardy}. 
Our interest is then to look for asymptotic estimates for the $p$-Green function of $(M,g)$ and its gradient so to better control growth at infinity of the weights in \eqref{eq:log hardy}. 
One possibility is to use Li-Yau type estimates which are ensured under several lower bounds on Ricci. 
These, however, are not sufficient because it provides only an upper bound on $\nabla \log{\mathcal{G}_p}$.

Thus, instead of using the $p$-Green function of $(M, g)$ directly, we use the $p$-Green function of the model manifold $(\widetilde{M}, \widetilde{g})$ constructed in \Cref{subs:model manifold} which is $p$-superharmonic outside a large enough compact set and whose estimates are already available.

Notice also that $G_p(x) \to 0$ as $r(x) \to + \infty$, hence, $G_p(x)$ distant form 1 provided that $r(x)>>1$ 
In other words, $G_p(x)$ is a suitable weight in \Cref{thm:log hardy} as long as $r(x)>>1$.

\begin{proposition}
\label{thm:hardy with p-green}
Let $(M, g)$ be a Cartan-Hadamard manifold satisfying \eqref{eq:upper bound ric}. 

For $p > 1$ and $\beta \geq 0$ there exists a compact $K$ containing the pole such that 
\begin{equation}
    \label{eq:hardy with p-green}
    \int_\Omega \frac{|\nabla G_p|^p}{|G_p|^p} (-\log{G_p})^{\beta p} |f|^p dV_g \leq \left(\frac{p}{p-1}\right)^p \int_\Omega (-\log{G_p})^{\beta p} |\nabla f |^p dV_g, 
\end{equation}
for all $f \in C^\infty_0(\Omega)$
where $\Omega = M \setminus K$.
\end{proposition}

Using estimates \eqref{eq:asymptotic G'} and \eqref{eq:asymptotic G} we deduce 
\begin{equation}
    \label{eq:asymptotic gradG/G}
    \frac{|\nabla G_p|}{|G_p|}(r(x)) \sim D_5 r(x)^{\frac{\alpha}{2}},
\end{equation}
\begin{equation}
    \label{eq:asymptotic log G}
    (-\log{G_p(r(x))}) \sim D_6 r(x)^{1+\frac{\alpha}{2}}
\end{equation}
so that 
\begin{equation}
    \label{eq:estimate log}
    |\log{G_p}|^\beta  = \mathcal{O}(|\nabla \log{G_p}|)
\end{equation}
provided that $\beta \leq \frac{\alpha}{2+\alpha}$. 

Note that \Cref{thm:hardy with p-green} requires $f$ to be smooth and compactly supported in $\Omega$. 
Both assumptions, however, can be weakened as long as the support of $f$ is far away from the pole $o$. 

\begin{theorem}
\label{thm:hardy with p-green sobolev}
Let $(M, g)$ be a Cartan-Hadamard manifold satisfying \eqref{eq:upper bound ric}. 
For $p > 1$ and $0\leq \beta \leq \frac{\alpha}{\alpha + 2}$ there exists a compact $K$ containing the pole such that 
\begin{equation}
    \label{eq:hardy with p-green sobolev}
    \int_M \frac{|\nabla G_p|^p}{|G_p|^p} (-\log{G_p})^{\beta p} |f|^p dV_g \leq \left(\frac{p}{p-1}\right)^p \int_M (-\log{G_p})^{\beta p} |\nabla f |^p dV_g, 
\end{equation}
for all $f \in W^{1, p}(M)$ with $\supp(f) \cap K = \emptyset$.
\begin{proof}
We proceed by steps, gradually weakening the assumptions on $f$. 
\begin{enumerate}[label=\textbf{Step \arabic{enumi}}, wide=0pt]
    \item We begin by considering $f \in W^{1, p}(M)$ compactly supported in $\Omega = M \setminus K$ so that $f \in W^{1, p}_0(\Omega)$, i.e., there exists $u_n \in C^\infty_0(\Omega)$ such that $u_n \to f$ in $W^{1, p}$ norm.
    Note that $\supp(u_n)$ and $\supp(f)$ are all contained in a compact $\Omega' \subset \Omega$. 
    Then, by \eqref{eq:hardy with p-green} we have
    \begin{equation}
        \label{eq:hardy with p-green approximation}
        \int_M \frac{|\nabla G_p|^p}{|G_p|^p} (-\log{G_p})^{\beta p} |u_n|^p dV_g \leq \left(\frac{p}{p-1}\right)^p \int_M (-\log{G_p})^{\beta p} |\nabla u_n |^p dV_g. 
    \end{equation}
    Note that
    \begin{equation*}
        \left\vert \int_M (-\log{G_p})^{\beta p} (|\nabla u_n |^p - |\nabla f |^p) dV_g \right\vert \leq \sup_{\Omega'}(-\log{G_p})^{\beta p} \int_M \vert|\nabla u_n |^p - |\nabla f |^p\vert dV_g,
    \end{equation*}
    so that
    \begin{equation*}
      \int_M (-\log{G_p})^{\beta p} |\nabla u_n |^p dV_g \to \int_M (-\log{G_p})^{\beta p} |\nabla f |^p dV_g. 
    \end{equation*}
    Similarly
    \begin{equation*}
        \int_M \frac{|\nabla G_p|^p}{|G_p|^p} (-\log{G_p})^{\beta p} |u_n|^p dV_g \to \int_M \frac{|\nabla G_p|^p}{|G_p|^p} (-\log{G_p})^{\beta p} |f|^p dV_g.  
    \end{equation*}
    Hence, passing to the limit in \eqref{eq:hardy with p-green approximation} we obtain the validity of \eqref{eq:hardy with p-green sobolev} for all $f \in W^{1, p}(M)$ compactly supported in $\Omega$.
    \item Next, let $f \in W^{1,p}(M)$ such that $\supp(f) \cap K = \emptyset$ and consider a family of cutoffs $\chi_R \in C^\infty(M)$ such that $\chi \equiv 1$ on $B_R$, $\chi_R \equiv 0$ outside $B_{2R}$ and $|\nabla \chi_R| \leq C$ uniformly on $R$.
    Such a family exists on any complete Riemannian manifold, see \cite{Ga1959}. 
    Consider $f\chi_R \in W^{1, p}(M)$, clearly $\supp(f\chi_R) \subseteq M\setminus K$ is compact. 
    Then, by Step 1 (with $\beta = 0$) we have
    \begin{align*}
        \begin{split}
        \int_M \frac{|\nabla G_p|^p}{|G_p|^p} |f|^p|\chi_R|^p dV_g &\leq \left(\frac{p}{p-1}\right)^p 2^{p-1} \left(\int_M |f|^p|\nabla \chi_R|^p dV_g + \int_M |\nabla f|^p|\chi_R|^p dV_g \right)\\
        &\leq \left(\frac{p}{p-1}\right)^p2^{p-1} \left(\int_M |\nabla f|^p dV_g + \int_{B_{2R}\setminus B_R}|f|^p dV_g\right).
        \end{split}
    \end{align*}
    Note that the LHS converges to $\int_M |\nabla \log G_p|^p |f|^p dV_g$ by monotone convergence, on the other hand $\int_{B_{2R}\setminus B_R}|f|^p dV_g \to 0$ since $f \in L^p(M)$. 
    We conclude that
    \begin{equation}
        \label{eq:hardy step 2}
        \int_M |\nabla \log G_p|^p |f|^p dV_g \leq \left(\frac{p}{p-1}\right)^p2^{p-1} \int_M |\nabla f |^p dV_g
    \end{equation}
    for all $f \in W^{1, p}(M)$ with $\supp \cap K = \emptyset$.
    \item Using Step 2, we now prove the more general \eqref{eq:hardy with p-green sobolev} under the assumptions that $f \in W^{1, p}(M)$ and $\supp(f) \cap K = \emptyset$. 
    Indeed, let $\chi_R \in C^\infty(M)$ be as in Step 2 so that $f\chi_R$ is compactly supported in $M \setminus K$, by Step 1 we have
    \begin{align*}
        \begin{split}
        \int_M |\nabla \log G_p|^p (-\log G_p)^{\beta p} |f|^p|\chi_R|^p dV_g \leq& \left(\frac{p}{p-1}\right)^p2^{p-1}\left(\int_M (-\log G_p)^{\beta p}|f|^p|\nabla \chi_R|^p dV_g \right.\\
        &+\left. \int_M (-\log G_p)^{\beta p}|\nabla f|^p|\chi_R|^p dV_g\right). 
        \end{split}
    \end{align*}
    Here, we reason as in Step 2. 
    The only difference is the following estimate which is a consequence of~\eqref{eq:estimate log} and \eqref{eq:hardy step 2}:
    \begin{equation*}
            \int_M (-\log G_p)^{\beta p}|f|^p dV_g \leq C\int_M |\nabla\log G_p|^p |f|^p dV_g \leq C \left(\frac{p}{p-1}\right)^p2^{p-1}\int_M |\nabla f|^p dV_g
    \end{equation*}
    where $C>0$.
    Since $|\nabla f | \in L^p(M)$ we are still able to conclude that $\int_{B_{2R}\setminus B_R} (-\log G_p)^{\beta p}|f|^p dV_g \to 0$ as $R \to +\infty$. 
\end{enumerate}
\end{proof}
\end{theorem}

If we require $f \in W^{2, p}(M)$ and apply \eqref{eq:hardy with p-green sobolev} twice, we obtain the following second order Hardy-type inequality.

\begin{theorem}
\label{thm:rellich with p-green sobolev}
Let $(M, g)$ be a Cartan-Hadamard manifold satisfying \eqref{eq:upper bound ric}. 
For $p > 1$ and $0 \leq \beta \leq \frac{\alpha}{2+\alpha}$ there exists a compact $K$ containing the pole such that 
\begin{equation}
    \label{eq:rellich with p-green sobolev}
    \int_M \frac{|\nabla G_p |^p}{|G_p|^p} (-\log{G_p})^{\beta p} |f|^p dV_g \leq C\int_M |\nabla^2 f |^p dV_g,
\end{equation}
for all $f \in W^{2, p}(M)$ such that $\supp(f) \cap K = \emptyset$, where and $C = C(p, K) > 0$.
\begin{proof}
Using \Cref{thm:hardy with p-green sobolev} and \eqref{eq:estimate log} we have
\begin{equation*}
    \int_M \frac{|\nabla G_p|^p}{|G_p|^p} (-\log{G_p})^{\beta p} |f|^p dV_g \leq C \int_M |\nabla \log G_p |^p |\nabla f |^p dV_g.
\end{equation*}
Since $|\nabla f | \in W^{1, p}(M)$ with $\supp(|\nabla f |) \cap K = \emptyset$ we apply \eqref{eq:hardy with p-green sobolev} with $\beta = 0$ to $|\nabla f |$ and conclude using Kato's inequality $|\nabla|\nabla f||\le |\nabla^2 f|$.
\end{proof}
\end{theorem}

Note that inequality \eqref{eq:rellich with p-green sobolev} is more of a second-order Hardy-type inequality rather then a proper Rellich inequality. 
The reason being that the RHS is estimated with the $L^p$-norm of the Hessian rather than the Laplacian of $f$. 
The optimal value for $\beta$ in \eqref{eq:rellich with p-green sobolev} is $\beta = \frac{\alpha}{2+\alpha}$, in this case we have:
\begin{equation*}
    \frac{|\nabla G_p|}{|G_p|}(-\log{G_p})^{\frac{\alpha}{2+\alpha}} \sim D_7 r(x)^\alpha
\end{equation*}
which is the fastest growth we are able to control via \eqref{eq:rellich with p-green sobolev}.
Finally we observe that no assumption on the support of $f$ is needed as long as the weight has support distant from the pole. This is the kind of control needed for applications. 

\begin{theorem}
\label{thm:rellich with weight sobolev}
Let $(M, g)$ be a Cartan-Hadamard manifold satisfying \eqref{eq:upper bound ric}. 
For $p > 1$ and $K$ as in \Cref{thm:rellich with p-green sobolev}, let $\omega \geq 0$ be a measurable function such that $\supp (\omega) \cap K = \emptyset$  and $\omega(x) = \mathcal{O}(r^\alpha(x))$ on $M$, then $W^{2, p}(M)\hookrightarrow L^p(M, \omega^p dV_g)$.
\begin{proof}
In order to extend the support of $f$, we need to remove the possible problems around the pole. 
To do so, let $K'$ a compact set such that $K \subseteq K' \subseteq M \setminus \supp(\omega)$, let $\varphi \in C^\infty(M)$ be a cutoff function such that $\varphi \equiv 0$ on $K$ and $\varphi \equiv 1$ outside of $K'$. 
Note that $|\nabla \varphi|$ and $|\nabla^2\varphi|$ are uniformly bounded and that $f\varphi \in W^{2, p}(M)$ with $\supp(f\varphi)\cap K = \emptyset$, then by \Cref{thm:rellich with p-green sobolev} we have
\begin{align*}
    \int_{M} \omega^p |f|^p dV_g &= \int_{\Omega} \omega^p |f\varphi|^p dV_g \leq C' \int_{M} \frac{|\nabla G_p |^p}{|G_p|^p} (-\log{G_p})^{\frac{\alpha}{2+\alpha} p} |\varphi f|^p dV_g \\
   &\leq C\int_\Omega |\nabla^2 (\varphi f) |^p dV_g \\
   &\leq C\int_\Omega |\nabla^2 f|^p dV_g + C\int_\Omega |\nabla \varphi|^p|\nabla f |^p dV_g + C\int_\Omega |\nabla^2 \varphi|^p |f|^p  dV_g \\
   & \leq C \Vert f \Vert^p_{W^{2, p}(M)}.
\end{align*}
\end{proof}
\end{theorem}

As a direct consequence, if we have a family of weights $\{\omega_R\}$ whose growth is suitably controlled and whose supports vanish at $+\infty$ then $\Vert \omega_R f \Vert_{L^p} \to 0$.
\begin{corollary}
\label{cor:rellich with weight hessian cutoffs}
Let $p > 1$ and $(M, g)$ as in \Cref{thm:rellich with weight sobolev}. 
Let $f \in W^{2,p}(M)$ and $\{\omega_R\}\subseteq C^\infty(M)$ non-negative such that $\supp(\omega_R)\subseteq M \setminus \overline{B}_R$ with $R>>1$ and $\omega_R(x) \leq C r^\alpha(x)$, then
\begin{equation*}
    \label{eq:rellich with hessian goes to 0}
    \int_{M} \omega_R^p |f|^p dV_g \to 0
\end{equation*}
as $R\to +\infty$. 
\end{corollary}

\begin{remark}
\label{rmk:Hardy for L^p pp}
Note that if we assume lower regularity in $f$, namely, $f\in W^{1, p}(M)$ we are still able to control $\Vert \omega f \Vert_{L^p(M)}$ as long as $\omega(x) \leq C r^{\frac{\alpha}{2}}(x)$. 
The strategy here is the same of \Cref{thm:rellich with weight sobolev} but instead of the second order Hardy \eqref{eq:rellich with p-green sobolev} we use the Hardy-type inequality \eqref{eq:hardy with p-green sobolev} with $\beta = 0$.
Similarly, if we take a family of weights $\{\omega_R\}$ such that $\omega_R(x) \leq C r^{\frac{\alpha}{2}}(x)$ and $\supp(\omega_R)\subseteq M \setminus \overline{B}_R$, we are still able to conclude that $\Vert \omega_R f \Vert_{L^p} \to 0$.
\end{remark}

\section{Density in \texorpdfstring{$W^{2,p}$}{Sobolev Spaces}}
\label{sec:density}
In the following section, we apply the estimates developed in \Cref{sec:Hardy and Rellich} to the density problem of smooth and compactly supported functions in the Sobolev space $W^{2, p}(M)$.
To this aim, we construct via the Riemannian distance a family of smooth cutoff functions $\{\chi_R\}$ which we control up to the second covariant derivative.
On arbitrary Riemannian manifolds there are two obstacles to this construction: the Riemannian distance might fail to be smooth on $M\setminus \{o\}$ and, while $|\nabla r|$ is always bounded, $|\nabla^2 r|$ might grow uncontrollably.
In the case of Cartan-Hadamard manifolds, however, both difficulties can be overcome.
Indeed, the cut locus of $M$ is empty which implies smoothness of the Riemannian distance. Furthermore, a lower bound on the radial Ricci curvature allows to control the Hilbert-Schmidt norm of $\nabla^2 r$. 

\begin{lemma}
\label{lem:upper bound hessian distance}
Let $(M, g)$ be a Cartan-Hadamard manifold satisfying
\begin{equation}
    \label{eq:lower bound ricci}
    \Ric_o(x) \geq -(n-1)B^2 r^\beta(x) \quad \forall x \in M\setminus B_{R_0}, 
\end{equation}
for some $B, R_0 > 0$ and $\beta \geq 0$. 
Then, there exist $R_1 > R_0$ and $C > 0$ such that 
\begin{equation}
    \label{eq:hessian distance function}
    |\nabla^2 r|(x) \leq C r^{\frac{\beta}{2}}(x) \quad  \forall x \in M \setminus B_{R_1}.
\end{equation}

\begin{proof}
By the Hessian comparison theorem (\cite[Theorem 2.3]{PRS2008}), the Hessian of the Riemannian distance, $\nabla^2 r$, has non negative eigenvalues at every point in $M\setminus B_{R_0}$ and in particular
\begin{equation*}
    |\nabla^2 r| \leq \Delta r 
\end{equation*}
on $M\setminus B_{R_0}$. 
Then, by Laplacian comparison we conclude that 
\begin{equation*}
    |\nabla^2 r| \leq (n-1) \frac{j'(r)}{j(r)},
\end{equation*}
where $j$ is smooth solution of
\begin{equation*}
    \begin{cases}
    j''(t) - B^2t^\beta j(t) = 0 \\
    j(0) = 0, \quad j'(0) = 1
    \end{cases}
\end{equation*}
on $[0, +\infty)$.
With similar estimates as in \Cref{subs:model manifold} we get $\frac{j'(t)}{j(t)} \sim B t^{\frac{\beta}{2}}$ for $t \to +\infty$, in particular, there exist some $R_1>R_0$ and some positive constant $\widetilde{C}$, depending on $B, \beta, R_1$ such that 
\begin{equation*}
    \label{eq:growth estimate j'/j}
     \frac{j'(t)}{j(t)} \leq \widetilde{C} t^{\frac{\beta}{2}}
\end{equation*}
for $t \geq R_1$. 
It follows that
\begin{equation*}
    |\nabla^2 r|(x) \leq (n-1)\frac{j'(r(x))}{j(r(x))} \leq C r^{\frac{\beta}{2}}(x), 
\end{equation*}
for all $x \in  M \setminus B_{R_1}$, where $C = C(B, \beta, R_1, n)$.
\end{proof}
\end{lemma}

Once we have second order estimates on the Riemannian distance, we obtain $\{\chi_R\}$ by composing with a sequence of real cutoffs. 

\begin{lemma}
\label{lem:existence of cutoffs}
Let $(M, g)$ be a Cartan Hadamard manifold satisfying \eqref{eq:lower bound ricci}, then, there exists a family of smooth cutoffs $\{\chi_R\}\in C^\infty_0(M)$ with $R>>1$ such that
\begin{enumerate}
    \item $\chi_R \equiv 1$ on $B_R$ and $\chi_R \equiv 0$ on $M \setminus \overline{B_{2 R}}$;
    \item $|\nabla \chi_R| \leq \frac{C_1}{R}$;
    \item $|\nabla^2 \chi_R| \leq C_2 R^{\frac{\beta}{2} -1}$, 
\end{enumerate}
with $C_1, C_2 > 0$.
\begin{proof}
Fix $\phi: \R \to [0, 1]$ a smooth function such that $\phi \equiv 1$ on $(-\infty, 1]$ and $\phi \equiv 0$ on $[2, + \infty)$, and let $a> 0$ such that $|\phi'| + |\phi''| \leq a$ uniformly on $\R$. For $R >> 1$ (it suffices $R \geq R_1$, with $R_1$ as in \Cref{lem:upper bound hessian distance}), let
\begin{equation*}
    \phi_R(t) \coloneqq \phi\left(\frac{t}{R}\right)
\end{equation*}
so that 
\begin{equation*}
    |\phi_R'|\leq \frac{a}{R}, \qquad |\phi_R''|\leq \frac{a}{R^2}.
\end{equation*}
Then, define $\chi_R(x) \coloneqq \phi_R \circ r(x)$, we have $\chi_R \equiv 1$ on $B_R$ and $\chi_R \equiv 0$ on $M \setminus B_{2 R}$. 
Furthermore, 
\begin{align*}
    |\nabla \chi_R| &\leq  |\phi_R'(r(x))|  |\nabla r(x)| \leq \frac{C_1}{R}\\
    |\nabla^2 \chi_R| &\leq |\phi_R'(r(x))| |\nabla^2 r(x)| + |\phi_R''(r(x))|  |\nabla r(x)|^2 \leq C_2 R^{\frac{\beta}{2} -1}, 
\end{align*}
where $C_1, C_2$ depend on $a$ and the constant $C$ of \Cref{lem:upper bound hessian distance}.
\end{proof}
\end{lemma}

\begin{remark}
\label{rmk:distance like case}
The above construction of the Hessian cutoffs is not the only possible one. 
It is worth noticing that the family $\{\chi_R\}$ can be constructed on Riemannian manifolds without any topological restrictions as long as one of the following assumptions holds: 
\begin{enumerate}[label = (\alph*)]
\item $\displaystyle|\Ric|(x) \leq B^2 r^{\beta}(x)$ and $\displaystyle\inj(x) \geq i_0 r^{-\frac{\beta}{2}}(x) > 0$
\item $\displaystyle|\Sect|(x) \leq B^2 r^{\beta}(x)$,
\end{enumerate}
for some $ B, i_0 > 0$ and $\beta \geq 0$.
In this setting, although the Riemannian distance might loose smoothness, it is possible to construct a distance-like function $H \in C^\infty(M)$ such that 
\begin{enumerate}[label=(\roman*)]
    \item $C^{-2} r(x) \leq H(x) \leq \max\{r(x), 1\}$;
    \item $|\nabla H(x)| \leq 1$;
    \item $|\nabla^2 H(x)| \leq C \max \{r^{\frac{\beta}{2}}(x), 1\}$,
\end{enumerate}
for some $C>1$, see \cite[Theorem 1.2]{IRV2020}. 
Then, one defines $\chi_R = \phi_R \circ H(x)$ where $\phi_R$ is a family of real cutoffs in a similar fashion of \Cref{lem:existence of cutoffs}. 
\end{remark}

We can now prove the density of smooth compactly supported functions in the Sobolev space $W^{2, p}$, namely (a) of \Cref{th_main}.
To obtain this we assume a double bound on the radial Ricci curvature. 
The bound from below allows the construction of the smooth cutoff functions while the bound from above ensures the validity of the functional estimates in \Cref{sec:Hardy and Rellich}. 

\begin{theorem}
\label{thm:density on CH manifolds}
Let $(M, g)$ be a Cartan-Hadamard manifold with a fixed pole $o\in M$. 
Suppose that
\begin{equation*}
    -(n-1)B^2 r^{2\alpha + 2}(x) \leq \Ric_o(x) \leq - 2(n-1)^2A^2 r^\alpha(x), \quad \forall x \in M\setminus B_{R_0}
\end{equation*}
for some $A, B, R_0 >0$ and $\alpha \geq 0$.
Then $W^{2, p}_0(M) = W^{2, p}(M)$ for all $p > 1$. 
\begin{proof}
Since $C^\infty(M)\cap W^{2, p}(M)$ is dense in $W^{2, p}(M)$ (see \cite{GGP2017}), it suffices to show that $C^\infty_0(M)$ is dense in $C^\infty(M)\cap W^{2, p}(M)$ with respect to the $W^{2, p}$ norm. 
To this goal, take $f \in C^\infty(M)\cap W^{2, p}(M)$ and consider a family of cutoffs $\{\chi_R\} \subseteq C^\infty(M)$ as in \Cref{lem:existence of cutoffs}. 
Define $f_R \coloneqq \chi_R f \in C^\infty_0(M)$ and observe that
\begin{align}
    \label{eq:Lp density function}
    \Vert (f_R - f)\Vert_{L^p} &= \Vert (\chi_R - 1) f\Vert_{L^p} \\
    \label{eq:Lp density gradient}
    \Vert \nabla(f_R - f)\Vert_{L^p} &\leq \Vert f\nabla\chi_R\Vert_{L^p} + \Vert (\chi_R - 1) \nabla f\Vert_{L^p} \\
    \label{eq:Lp density hessian}
    \Vert \nabla^2(f_R - f)\Vert_{L^p} &\leq 2\Vert |\nabla f||\nabla\chi_R|\Vert_{L^p} + \Vert (\chi_R - 1) \nabla ^2 f\Vert_{L^p}  + \Vert  f \nabla ^2 \chi_R \Vert_{L^p}. 
\end{align}
Since $\nabla \chi_R$ and $(\chi_R - 1)$ are uniformly bounded and supported in $M\setminus \overline{B}_R$, $f \in W^{2, p}(M)$ implies that the RHS of \eqref{eq:Lp density function}, \eqref{eq:Lp density gradient}, and \eqref{eq:Lp density hessian} except the last term, vanish as $R \to + \infty$.
We only need to show that $|| f \nabla^2 \chi_R||_{L^p} \to 0$ as $R\to 0$.
To see this, it is sufficient to observe that $|\nabla^2\chi_R| \leq C r^\alpha$ and $\supp(\chi_R) \subseteq M \setminus \overline{B}_R$, then by \Cref{cor:rellich with weight hessian cutoffs} we conclude the proof. 
\end{proof}
\end{theorem}

\begin{remark}
When $p = 1$ our strategy to construct Hardy-type inequalities fails. 
Note for instance that the constant in \eqref{eq:log hardy} and subsequent derived inequalities explodes as $p \to 1$. 
Nevertheless, we expect the result to hold even when $p=1$. 
\end{remark}

\section{An \texorpdfstring{$L^2$}{L^2}-Calder\'on-Zygmund inequality}
\label{sec:CZ(2)}
As a further application of the tools developed in \Cref{sec:Hardy and Rellich}, we prove the validity of a $L^2$-Calder\'on-Zygmund inequality on Cartan-Hadamard manifolds with bounds on Ricci curvature. 
In the spirit of \cite{IRV2019}, we first prove a weighted $CZ(2)$ inequality which holds under lower bound on Ricci curvature.

\begin{theorem}
\label{thm:CZ(2) weighted}
Let $(M, g)$ be a Cartan-Hadamard manifold with a fixed pole $o \in M$.
Suppose that 
\begin{equation*}
    \Ric (x) \geq - (n-1)B^2 r^\beta(x) \quad \forall x \in M \setminus B_{R_0}
\end{equation*}
in the sense of quadratic forms for some $B, R_0$ and $\beta \geq 0$. 
Then, for every $\varepsilon > 0$ there exists a constant $A_1 = A_1(\varepsilon) > 0$ such that 
\begin{equation}
    \label{eq:CZ(2) weighted}
    \Vert \nabla^2\varphi \Vert_{L^2} \leq A_1 \left[ \Vert \Delta \varphi \Vert_{L^2} + \Vert \varphi \Vert_{L^2} \right] + A_2\varepsilon^2 \Vert r^\beta \varphi \Vert_{L^2} \qquad \forall \varphi \in C^\infty_0(M).
\end{equation}
Here $A_2$ is a fixed positive constant independent of $\varepsilon$. 

\begin{proof}
Let $R_1$ be the constant of \Cref{lem:upper bound hessian distance} and let $h: \R \to [1, + \infty)$ be a smooth function such that $h(t) \equiv 1$ for $t \leq R_1$, and $h(t) = t - a$ for some $a>0$ and $t \geq 2R_1$ and $|h'(t)| \leq 1$ for all $t$.
Define $H(x) \coloneqq h(r(x))$ so that
\begin{enumerate}[label=(\roman*)]
    \item $\max\{C^{-1}r(x), 1\} \leq H(x) \leq \max\{Cr(x), 1\}$;
    \item $|\nabla H (x)| \leq 1$;
    \item $|\nabla^2 H(x)| \leq C r^{\frac{\beta}{2}}(x)$
\end{enumerate}
for some $C>1$.
Consider on $(M, g)$ the following conformal deformation:
\begin{equation*}
    \tilde{g} \coloneqq e^{2\phi} g, \text{ where } \phi \coloneqq \frac{\beta}{2} \log{H}. 
\end{equation*}
Note that $M$ remains complete also with respect to $\tilde g$. 
In this proof we denote with $\widetilde{\nabla}^2, \widetilde{\nabla}, \widetilde{\Delta}$, $\widetilde{\Ric}$ the Hessian, gradient, Laplacian and Ricci tensor with respect to $\widetilde{g}$. 
Moreover, let $\widetilde{L^2} = L^2(M, dV_{\tilde{g}})$.
Since 
\begin{equation*}
    |\nabla \phi| \leq \frac{\beta}{2}, \qquad |\nabla^2 \phi| \leq \beta\max\{C r^{\frac{\beta}{2}}(x), 1\}, 
\end{equation*}
we immediately get that $\widetilde{\Ric}$ is bounded from below by some constant $\widetilde{C}$ depending on $\beta, B, C$ and $n$, see (25) in \cite{IRV2019}. 
Thanks to \cite[Proposition 4.5]{GP2015}, this implies the validity on $(M, \widetilde{g})$ of the following \textit{infinitesimal} $CZ(2)$ inequality: for every $\varepsilon > 0$
\begin{equation}
    \label{eq:CZ(2) epsilon}
    \Vert \widetilde{\nabla}^2 u\Vert_{\widetilde{L^2}}^2 \leq \frac{\widetilde{C} \varepsilon^2}{2} \Vert u \Vert_{\widetilde{L^2}}^2 + \left(1 + \frac{\widetilde{C}^2}{2\varepsilon^2}\right) \Vert \widetilde{\Delta}u\Vert_{\widetilde{L^2}}^2 \qquad \forall u \in C^\infty_0(M).
\end{equation} 
Throughout the rest of the proof, we denote with $A_i$ real positive constants depending on $\beta, B, n, C$ and, possibly, $\varepsilon$.
By standard estimates (see \cite[Section 8.3]{IRV2019}) we get
\begin{align}
\label{eq:hessian conformal estimate}
|\widetilde{\nabla}^2 u|^2 dV_{\widetilde{g}} &\geq e^{(n-4)\phi} \left\lbrace \frac{1}{2}|\nabla^2 u|^2 - A_3|\nabla u|^2 - |\Delta u|^2 \right\rbrace dV_g,\\
\label{eq:laplacian conformal estimate}
|\widetilde{\Delta}u|^2 dV_{\widetilde{g}} &\leq e^{(n-4)\phi} \left\lbrace |\Delta u|^2 + A_4|\nabla u|^2 \right\rbrace dV_g. 
\end{align}
Inserting \eqref{eq:hessian conformal estimate} and \eqref{eq:laplacian conformal estimate} in \eqref{eq:CZ(2) epsilon} yields
\begin{equation}
    \label{eq:CZ(2) epsilon conformal}
    \Vert H^{\frac{\beta}{4}(n-4)} |\nabla^2 u|\Vert^2_{L^2} \leq \widetilde{C}\varepsilon^2 \Vert H^{\frac{\beta}{4}n} u\Vert^2_{L^2} + A_5\Vert H^{\frac{\beta}{4}(n-4)} \Delta u\Vert^2_{L^2} + A_6 \Vert H^{\frac{\beta}{4}(n-4)} |\nabla
     u|\Vert^2_{L^2}.
\end{equation}
Note here that $A_5$ and $A_6$ depend also on $\varepsilon$.
For $\varphi \in C^\infty_0(M)$ we take $u = H^{-(n-4)\frac{\beta}{4}}\varphi \in C^{\infty}_0(M)$. 
By straightforward computations we obtain
\begin{align*}
    H^{(n-4)\frac{\beta}{4}}\nabla u =& -(n-4)\frac{\beta}{4} \frac{\nabla H}{H} \varphi + \nabla \varphi,\\
     H^{(n-4)\frac{\beta}{4}}\nabla^2 u =& \nabla^2 \varphi -(n-4)\frac{\beta}{4} \left(\frac{\nabla H}{H} \otimes \nabla \varphi + \nabla\varphi \otimes \frac{\nabla H}{H}\right) \\
     &+ \varphi \left\{ -(n-4)\frac{\beta}{4} \frac{\nabla^2 H}{H} +(n-4)\frac{\beta}{4}\left[(n-4)\frac{\beta}{4}+1\right] \frac{\nabla H}{H} \otimes \frac{\nabla H}{H} \right\}, \\
     H^{(n-4)\frac{\beta}{4}}\Delta u =& \Delta \varphi -(n-4)\frac{\beta}{4} g\left(\frac{\nabla H}{H}, \nabla \varphi\right) \\
     &+ \varphi \left\{ -(n-4)\frac{\beta}{4} \frac{\nabla^2 H}{H} +(n-4)\frac{\beta}{4}\left[(n-4)\frac{\beta}{4}+1\right] g\left(\frac{\nabla H}{H}, \frac{\nabla H}{H}\right) \right\}.
\end{align*}
Since 
\begin{equation*}
    \frac{|\nabla H|}{|H|} \leq 1 \qquad \frac{|\nabla^2 H|}{|H|} \leq C^2 r^{\frac{\beta}{2}-1}, 
\end{equation*}
we have
\begin{align*}
     H^{(n-4)\frac{\beta}{4}}|\nabla u| \leq& A_7|\varphi| + |\nabla \varphi|;\\
      H^{(n-4)\frac{\beta}{4}}|\nabla^2 u| \geq& |\nabla^2 \varphi| - A_7|\nabla \varphi| - A_8 H^{\frac{\beta}{2}-1}|\varphi|; \\
      H^{(n-4)\frac{\beta}{4}}|\Delta u| \leq& |\Delta \varphi| + A_7|\nabla \varphi| + A_8 H^{\frac{\beta}{2}-1}|\varphi|.
\end{align*}
Using these estimates in \eqref{eq:CZ(2) epsilon conformal} yields
\begin{equation}
    \label{eq:CZ(2) weighted intermediate}
    \Vert \nabla^2 \varphi\Vert^2_{L^2} \leq \widetilde{C}\varepsilon^2 \Vert H^{\beta} \varphi\Vert^2_{L^2} + A_9\Vert \nabla \varphi \Vert^2_{L^2} + A_5 \Vert \Delta
     u\Vert^2_{L^2} + A_{10} \Vert
     H^{\frac{\beta}{2}-1}\varphi\Vert^2_{L^2}.
\end{equation}
By the divergence theorem and Cauchy-Schwarz inequality we also have
\begin{equation*}
    \Vert \nabla \varphi \Vert_{L^2}^2 = \int_M |\nabla \varphi |^2 dV_g = -\int_M \varphi \Delta \varphi dV_g \leq 2\Vert \varphi \Vert_{L^2}^2 + 2\Vert \Delta \varphi \Vert_{L^2}^2.
\end{equation*}
Moreover, since $H^{\frac{\beta}{2}-1} = o(H^{\beta})$ as $r(x) \to +\infty$, for all $\varepsilon > 0$ there exists some constant $C_\varepsilon > 0$ such that 
\[H^{\frac{\beta}{2}-1} \leq\varepsilon \sqrt{\frac{\widetilde{C}}{A_{10}}} H^{\beta} + C_\varepsilon,\] 
hence, $A_{10} \Vert H^{\frac{\beta}{2}-1}\varphi\Vert^2_{L^2} \leq \varepsilon^2 \widetilde{C}\Vert H^{\beta} \varphi\Vert^2_{L^2} + C_\varepsilon^2 \Vert \varphi \Vert_{L^2}^2$.
Using these latter estimates, \eqref{eq:CZ(2) weighted intermediate} becomes
\begin{equation}
    \label{eq:CZ(2) weighted H}
    \Vert \nabla^2\varphi \Vert^2_{L^2} \leq A_1^2 \left[ \Vert \Delta \varphi \Vert^2_{L^2} + \Vert \varphi \Vert^2_{L^2} \right] + \widetilde{C}\varepsilon^2 \Vert H^\beta \varphi \Vert^2_{L^2}
\end{equation}
Finally, since $H(x) \leq \max\{C r(x), 1\}$ we have 
\begin{equation*}
    \int_M H^{2\beta}\varphi^2 dV_g \leq \int_{r \le 1} \varphi^2 dV_g + C^{2\beta}\int_{r\ge 1} r^{2\beta}\varphi^2 dV_g =  \Vert \varphi \Vert_{L^2} + C^{2\beta} \Vert r^\beta \varphi \Vert_{L^2}.
\end{equation*}
which gives \eqref{eq:CZ(2) weighted}
\end{proof}
\end{theorem}

\begin{remark}
Note that in \Cref{thm:CZ(2) weighted} we require a bound on Ricci in the sense of quadratic forms, that is 
\begin{equation*}
    \Ric(X, X)(x) \geq -(n-1)B^2r^\beta(x) g(X,X)
\end{equation*}
for any $X \in T_xM$. 
This is a stronger assumption than the previous bounds on radial Ricci curvature and is necessary to ensure the validity of \eqref{eq:CZ(2) epsilon}.  
\end{remark}

If we also assume also an upper bound on the Ricci curvature, using the second order Hardy-type inequality \eqref{eq:rellich with p-green sobolev} we can estimate the last term on \eqref{eq:CZ(2) weighted} thus proving (b) of \Cref{th_main}. 

\begin{theorem}
\label{thm:CZ(2) on CH}
Let $(M, g)$ be a Cartan-Hadamard manifold with a fixed pole $o \in M$. Suppose that
\begin{equation}
    - (n-1)B^2 r^\alpha(x) \leq \Ric(x) \leq -2(n-1)^2A^2 r^{\alpha}(x) \quad \forall x \in M \setminus B_{R_0}
\end{equation}
for some constants $B>\sqrt{2}(n-1)A>0$ and some $\alpha\ge 0$.
Then, the following $L^2$-Calder\'on-Zygmund inequality holds on $M$:
\begin{equation}
    \label{eq:CZ(2)}
    \Vert \nabla^2 \varphi \Vert_{L^2} \leq C \left(\Vert \Delta \varphi \Vert_{L^2} + \Vert \varphi \Vert_{L^2}\right)
\end{equation}
for all $\varphi \in C^\infty_0(M)$.
\begin{proof}
By \Cref{thm:CZ(2) weighted} we have the validity of \eqref{eq:CZ(2) weighted}, thus, we only need to estimate the weighted term $\Vert r^\beta \varphi\Vert_{L^2}^2$. 
Let $K$ be a compact large enough (see \Cref{thm:rellich with weight sobolev}), then
\begin{equation*}
    \Vert r^\beta \varphi\Vert_{L^2}^2 = \int_M r^{2\beta} \varphi^2 dV_g \leq \max_{K} r^{2\beta}\int_K \varphi^2 dV_g + \int_{M\setminus K} r^{2\beta} \varphi^2 dV_g.
\end{equation*}
Thanks to \Cref{thm:rellich with weight sobolev} we have 
\begin{equation*}
    \int_{M\setminus K} r^{2\beta} \varphi^2 dV_g \leq C'\int_M |\nabla^2 \varphi|^2 dV_g,
\end{equation*}
so that
\begin{equation*}
   \Vert \nabla^2 \varphi \Vert_{L^2} \leq A' \left(\Vert \Delta \varphi \Vert_{L^2} + \Vert \varphi \Vert_{L^2}\right) + A'' \varepsilon^2 \left(\Vert \varphi\Vert_{L^2}^2 + \Vert \nabla^2 \varphi\Vert_{L^2}^2\right).
\end{equation*}
Since $\varepsilon$ can be made arbitrarily small and $A''$ is a fixed constant this last estimate yields \eqref{eq:CZ(2)}. 
\end{proof}
\end{theorem}
\begin{remark}
It would be interesting to obtain a $CZ(p)$ estimate also in the general case $p \in(1,+\infty)$. 
To do this, however, one would need an \textit{infinitesimal} $CZ(p)$ estimate similar to \eqref{eq:CZ(2) epsilon} which, to the best of our knowledge, is not known when $p \neq 2$. 
\end{remark}
\begin{remark}
\label{rmk:CZ(p) implies equality of sobolev}
The validity of an $L^2$-Calder\'on-Zygmund inequality directly implies the density of $C^\infty_0(M)$ in $W^{2, 2}(M)$. 
The observation is due to S. Pigola and goes as follows. 
Let $\varphi \in W^{2, 2}(M)\subseteq H^{2, 2}(M)$, thanks to a result by O. Milatovic \cite[Appendix A]{GP2019}, there exists a sequence of functions $\{\varphi_k\}\subseteq C^\infty_0(M)$ such that $\varphi_k \to \varphi$ in $H^{2, 2}(M)$. 
It follows that $\{\varphi_k\}$ is Cauchy in $H^{2, 2}(M)$, using \eqref{eq:CZ(2)} and the validity on $(M, g)$ of an $L^2$-gradient estimate (see \cite[Proposition 3.10b]{GP2015}) we deduce that  $\{\varphi_k\}$ is Cauchy also in $W^{2, 2}(M)$. 
By completeness we have $\varphi_k \to \overline{\varphi}$ in $W^{2, 2}(M)$, however, $\overline{\varphi} = \varphi$ thanks to the continuous embedding $W^{2, 2}(M) \subseteq H^{2, 2}(M)$. 
See also \cite[Remark 2.1]{V2020}. 
As a result, \Cref{thm:CZ(2) on CH} provides an alternative proof of \Cref{thm:density on CH manifolds}, although under heavier assumptions.
\end{remark}

The above observation also implies the following corollary.

\begin{corollary}
Let $(M, g)$ be a Cartan-Hadamard manifold as in \Cref{thm:CZ(2) on CH}, then
\begin{equation*}
    W^{2, 2}_0(M) = W^{2, 2}(M) = H^{2,2}(M). 
\end{equation*}
\end{corollary}

\section{\texorpdfstring{$L^p$}{L^p}-positivity preserving and the BMS conjecture}
\label{sec:L^p pp and BMS}

This last section is devoted to the study of the $L^p$-positivity preserving property, and BMS conjecture, on a certain class of manifolds. 
Following \cite[Theorem B.1]{BMS02}, we first prove. 
\begin{lemma}
\label{lem:Lp pp smooth}
Let $\phi \in C^\infty_0(M)$, $\phi \geq 0$, then there exists a unique $v \in C^\infty (M) \cap L^p(M)$ $\forall p\in[1,+\infty]$, $v > 0$, such that
\begin{equation}
    \label{eq:Lp pp smooth}
    (-\Delta + 1)v = \phi.
\end{equation}

\begin{proof}
Let $\{\Omega_k\}$ be an exhaustion of $M$ by relatively compact, open sets of smooth boundary satisfying 
\begin{equation*}
    \Omega_1 \Subset \Omega_2 \Subset \cdots \Subset \Omega_k \Subset \Omega_{k+1}\Subset \cdots,
\end{equation*}
that is, $\Omega_k$ is relatively compact in $\Omega_{k+1}$ for all $k \in \mathbb{N}$.
Furthermore, assume $\Omega_1$ is large enough so that $\supp(\phi) \subseteq \Omega_1$. 
Then let $v_k$ be a smooth solution of the following Dirichlet problem:
\begin{equation}
\label{eq:dirichlet problem Lppp}
    \begin{cases}
    (-\Delta + 1)v_k = \phi &\text{ on }\Omega_k\\
    v_k = 0 &\text{ on }\partial\Omega_k.
    \end{cases}
\end{equation}
By the strong maximum principle (\cite[Theorem 3.5]{GT2001}), we immediately get $v_k > 0$ in the interior of $\Omega_k$ and $v_{k+1}\geq v_k$ for all $k$, hence, $\{v_k\}$ is a monotone increasing sequence of functions and thus admits a (possibly infinite) pointwise limit
\begin{equation*}
    0 < v(x) = \lim_{k \to + \infty} v_k(x).
\end{equation*}
Next we prove that $v$ is actually everywhere finite, smooth and belongs to $L^p(M)$ for any $p \in [1, + \infty)$. 
To this end, we multiply \eqref{eq:dirichlet problem Lppp} by $v_k^{p-1}$ and integrate over $\Omega_k$
\begin{align*}
    \int_{\Omega_k} v_k^{p-1} (-\Delta + 1)v_k dV_g &= \int_{\Omega_k} v_k^p dV_g -  \int_{\Omega_k} v_k^{p-1}\Delta v_k dV_g  \\
    &= \int_{\Omega_k} v_k^p dV_g + \int_{\Omega_k} \langle\nabla v_k^{p-1}, \nabla v_k \rangle dV_g \\
    &= \int_{\Omega_k} v_k^p dV_g + (p-1)\int_{\Omega_k} v_k^{p-2} |\nabla v_k|^2 dV_g \geq \int_{\Omega_k} v_k^p dV_g. 
\end{align*}
On the other hand, by H\"older's inequality
\begin{equation*}
    \int_{\Omega_k} v_k^{p-1} (-\Delta + 1)v_k dV_g = \int_{\Omega_k} v_k^{p-1} \phi dV_g \leq \left\lbrace\int_{\Omega_k} \phi^p dV_g\right \rbrace^{\frac{1}{p}} \left\lbrace\int_{\Omega_k} v_k^p dV_g\right \rbrace^{\frac{p-1}{p}}
\end{equation*}
hence $\Vert v_k \Vert_{L^p(\Omega_k)} \leq \Vert \phi \Vert_{L^p(M)}$. 
Since $\{v_k\}$ is uniformly bounded in $L^p$ on any compact set, by standard interior regularity we deduce that $\{v_k\}$ is uniformly bounded in $W^{h, p}_{\operatorname{loc}}(M)$ for any order $h$ and $p \in [1, + \infty)$.
As a consequence of Sobolev spaces compact embedding, all the covariant derivatives of $\{ v_k\}$ converge up to a subsequence uniformly on compact sets, i.e., $v_k$ converges in $C^\infty(M)$ topology. 
In particular $v$ is positive, smooth and satisfies \eqref{eq:Lp pp smooth}. Moreover, by Fatou's lemma we also have that $v \in L^p(M)$ for any $p \in [1, +\infty)$.  
For $p = +\infty$, let $x^\ast$ be such $v_k(x^\ast) = \max_{\Omega_K} v_k$, by the maximum principle we get $v_k(x^\ast) \leq \phi(x^\ast)\leq  \Vert \phi \Vert_{L^\infty(M)}$, that is, $\Vert v_k \Vert_{L^\infty(\Omega_K)} \leq \Vert \phi \Vert_{L^\infty(M)}$. Letting $k \to +\infty$ we get $v \in L^\infty(M)$. 
\end{proof}
\end{lemma}

We would like now to extend the above result to the case where $(-\Delta+1)v$ is positive Radon measure. 
To do so, we first need the following $L^q$-gradient estimate which is a simple extension of a result by
T. Coulhon and X. T. Duong, \cite{CD2003}. 

\begin{lemma}
\label{lem:Lp gradient estimates}
Let $(M, g)$ be a complete Riemannian manifold, then for all $1<q\le 2$ there exists a constant $C>0$ such that
\begin{equation}
    \label{eq:Lp gradient estimate}
    \Vert \nabla u \Vert_{L^q} \leq C(\Vert u \Vert_{L^q} + \Vert \Delta u \Vert_{L^q})
\end{equation}
for all $u \in C^\infty(M) \cap H^{2,q}(M)$.
\begin{proof}
The validity of \eqref{eq:Lp gradient estimate} on $C^\infty_0(M)$ is known thanks to a result by T. Coulhon and X. T. Duong \cite{CD2003}. 
Thanks to a result by O. Milatovic \cite[Appendix A]{GP2019}, for all $u \in C^\infty(M) \cap H^{2,q}(M)$ there exists a sequence of functions $\{u_k\} \subseteq C^\infty_0(M)$ such that $u_k \to u$ with respect to the $H^{2,q}(M)$ norm.
Applying \eqref{eq:Lp gradient estimate} to the Cauchy differences we deduce that $\nabla u_k \to \nabla u$ in $L^q(M)$. 
Then we obtain the desired result applying \eqref{eq:Lp gradient estimate} to $u_k$ and taking the limit. 
\end{proof}
\end{lemma}

\begin{theorem}
\label{thm:L^p pp Cartan hadamard}
Let $(M, g)$ be a Cartan-Hadamard manifold satisfying
\begin{equation*}
    - (n-1)B^2 r^{\alpha+2}(x) \leq \Ric_o(x) \leq -2(n-1)^2A^2 r^{\alpha}(x) \quad \forall x \in M\setminus B_{R_0}
\end{equation*}
for some constants $B>\sqrt{2}(n-1)A>0$ and some $\alpha, R_0 > 0$.
Then, $M$ is $L^p$-positivity preserving for all $2 \le p < +\infty$
\end{theorem}
\begin{proof}
Let $u \in L^p(M)$ such that $(-\Delta + 1)u \geq 0$ in the sense of distributions.
In order to prove $L^p$-positivity preserving, we need to show that
\begin{equation*}
    \int_M \phi u dV_g \geq 0 \qquad \forall \phi \in C^\infty_0(M), \phi \geq 0.
\end{equation*}
By \Cref{lem:Lp pp smooth}, let $v\in C^\infty(M)$, $v > 0$ such that $(-\Delta + 1)v = \phi$ and let $\{\chi_R\}\in C^\infty(M)$ be a family of cutoffs as in \Cref{lem:existence of cutoffs}.
Since $v\chi_R \in C^\infty_0(M)$, $v\chi_R \geq 0$ we have
\begin{align*}
    0 \leq \int_M u (-\Delta +1)(v\chi_R) dV_g =& \int_M \left[-u \Delta(v\chi_R) + v\chi_R u \right]dV_g \\
    =& - \int_{M} u \chi_R \Delta v dV_g -  \int_{M} u v \Delta \chi_R  dV_g \\ 
    &- \int_{M} u \langle \nabla \chi_R, \nabla v\rangle dV_g + \int_{M} u \chi_R v dV_g. 
\end{align*}
Recall that $v, \Delta v\in L^q(M)$ for all $q \in [1, +\infty]$, in particular, this holds for $q = \frac{p}{p-1}$ so that $uv \in L^1(M)$ and $u\Delta v \in L^1(M)$.
By dominated convergence we conclude that
\begin{equation*}
    \int_{M} u \chi_R \Delta v dV_g \to \int_{M} u \Delta v dV_g, \qquad \int_{M} u \chi_R v dV_g \to \int_{M} u v dV_g
\end{equation*}
for $R \to + \infty$. 
On the other hand, by \Cref{lem:Lp gradient estimates} we have $|\nabla v| \in L^q(M)$ so that $u\nabla v \in L^1(M)$, by dominated convergence we conclude that
\begin{equation*}
    \int_{M} u \langle \nabla \chi_R, \nabla v\rangle dV_g \leq  \int_{M} |u| |\nabla v| |\nabla \chi_R| dV_g \to 0
\end{equation*}
for $R \to +\infty$. 
Finally, by Holder's inequality we have 
\begin{equation*}
    \left| \int_{M} u v \Delta \chi_R  dV_g \right| \leq \left\lbrace\int_{M} |u|^p  dV_g\right\rbrace^{\frac{1}{p}}\left\lbrace\int_{M} |v\Delta \chi_R|^q   dV_g\right\rbrace^{\frac{1}{q}}. 
\end{equation*}
The lower bound on Ricci implies that $|\Delta \chi_R| \leq C r^{\frac{\alpha}{2}}(x)$ and $v \in W^{1, q}(M)$, hence, by \Cref{rmk:Hardy for L^p pp} we have
\begin{equation*}
   \int_{M} |v\Delta \chi_R|^q   dV_g \to 0
\end{equation*}
as $R \to + \infty$. 
In conclusion, we have proved that  
\begin{equation*}
    \int_M \phi u dV_g = \lim_{R \to + \infty}\int_M u (-\Delta +1)(v\chi_R) dV_g \geq 0,
\end{equation*}
hence, $u \geq 0$ in the sense of distributions. 
\end{proof}

Note that, although \Cref{lem:Lp pp smooth} holds on the whole $L^p$ scale, the case $p = +\infty$ and $1\leq p < 2$ have been left out in the previous theorem.
The difficulty in these situations is that we generally lack the $L^q$-gradient estimates where $q>2$ is the conjugate exponent of $p$. 
On the other hand, an $L^1$ gradient estimate which corresponds to the case $p = +\infty$ is false even in the Euclidean setting. 
\begin{remark}
\label{rmk:stoch compl}
Recall that the $L^\infty$-positivity preserving property implies stochastic completeness of the manifold at hand. 
Indeed, $(M,g)$ is stochastically complete if the only non-negative, bounded $C^2$ solution of
\begin{equation*}
    \Delta u = u
\end{equation*}
is $u \equiv 0$. 
We refer to \cite[Section 6]{Gr1999} for a survey of the equivalent definitions of stochastic completeness. 
Indeed, if $u \in C^2(M)\cap L^\infty(M)$, $u \geq 0$ solves $\Delta u = u$ then $-u$ solves $-\Delta (-u) -u \geq 0$. 
By $L^\infty$-positivity preserving we deduce that $u \leq 0$ hence $u \equiv 0$.
As a matter of fact, when the Ricci curvature is below a certain critical growth we can prove that $(M, g)$ looses stochastic completeness, hence, the $L^\infty$-positivity preserving property needs to fail. 
\end{remark}

\begin{theorem}
\label{thm:comparison stochastic completeness}
Let $(M, g)$ be a Cartan-Hadamard manifold satisfying 
\begin{equation*}
    \Ric_o(x) \le -2(n-1)^2A^2 r^\alpha(x) \quad \forall x \in M \setminus B_{R_0}
\end{equation*}
with $A, \alpha, R_0>0$. If $\alpha>2$, then $(M, g)$ is not stochastically complete.
\begin{proof}
Let $j$ be as in \eqref{eq:warping function model} and define
\begin{equation*}
    v(t) = \int_0^t j^{(1-n)}(s) \left(\int_0^s j^{(n-1)}(\tau) d\tau\right)ds
\end{equation*}
then $u(x) = v(r(x))$ is a $C^2$ function on $M$. 
By \eqref{eq:asymptotic stoch compl} we have 
\begin{equation*}
    \frac{\int_0^s j^{n-1}(\tau) d\tau}{j^{(n-1)}(s)} \in L^1(+\infty), 
\end{equation*}
hence, $u$ is bounded. 
Since $v' \geq 0$, by \Cref{prop:comparison CH} we have 
\begin{equation*}
    \Delta u(x) = v''(r(x)) + \Delta r(x) v'(r(x)) \geq v''(r(x)) + (n-1)\frac{j'(r(x))}{j(r(x))} v'(r(x))
\end{equation*}
for $r>>1$. 
By direct computation this implies that $\Delta u \geq 1$ outside a compact set and in particular, there cannot be a sequence of points $\{x_k\}\subset M$ such that $u(x_k)$ converges to $\sup_M u$ and $\Delta u(x_k) < 1/k$ which is an equivalent formulation of stochastic completeness, see \cite{PRS2003}. 
\end{proof}
\end{theorem}

It remains to investigate the subquadratic case.
In this setting the cut-off functions constructed in \Cref{lem:existence of cutoffs} have a much better behavior, namely, the Hessian is uniformly bounded. 
As a consequence, one can easily avoid the use of the Hardy-type inequality to control the term containing $\Delta \chi_R$. 
It turns out that such Laplacian cut-off functions exist on arbitrary complete Riemannian manifolds as long as the Ricci curvature satisfies
\begin{equation}
    \label{eq:ricci lambda}
    \Ric(x) \geq - \lambda^2(r(x)) \quad \forall x \in M \setminus B_{R_0}.
\end{equation}
Here $\lambda$ is a $C^\infty$ function given by
\begin{equation}
    \label{eq:lambda}
    \lambda(t) = \alpha t \prod_{j = 0}^k \log^{[j]}(t)
\end{equation}
for $t$ large enough, where $\alpha>0$, $k\in\mathbb{N}$ and $\log^{[j]}(t)$ stands for the $j$-th iterated logarithm. 
The following is a joint result of the second author with D. Impera and M. Rimoldi \cite[Corollary 4.1]{IRV2020}, which slightly generalizes \cite[Corollary 2.3]{BS2018}.

\begin{theorem}
\label{thm:cutoffs of IRV}
Let $(M, g)$ be a complete Riemannian manifold satisfying \eqref{eq:ricci lambda} in the sense of quadratic forms. 
Then, there exists a family of smooth cut-off functions $\{\chi_R\}\subseteq C^\infty_0(M)$, $R > R_0$, such that
\begin{enumerate}
    \item $\chi_R \equiv 1$ on $B_R$ and $\chi_R \equiv 0$ on $M \setminus \overline{B_{\gamma R}}$;
    \item $|\nabla \chi_R|\leq \frac{C_1}{\lambda(R)}$;
    \item $|\Delta \chi_R|\leq C_2$;
\end{enumerate}
where $C_1, C_2 > 0$, $\gamma > 1$ and $\lambda$ is the function defined in \eqref{eq:lambda}. 
\end{theorem}

Using these cut-off functions instead of the ones constructed in Lemma \ref{lem:existence of cutoffs} allows to drop any topological assumptions on $M$. This was already observed by B.  G\"uneysu in the setting where $L^p$-gradient estimates are available, i.e. for $p\in[2,\infty)$.
However, there is no need to use $L^p$-gradient estimates, since we can use instead a uniform Li-Yau estimate which is a special case of a result by D. Bianchi and A. Setti \cite[Theorem 2.8]{BS2018}.

\begin{theorem}
\label{thm:Li-Yau of Bianchi and Setti}
Let $(M, g)$ be a complete Riemannian manifold satisfying \eqref{eq:ricci lambda} in the sense of quadratic forms. 
Let $R > r > 0$ and let $\gamma > 1$ and let $v : M\setminus\overline{B_r} \to \R$ be a $C^2$ function satisfying
\begin{equation}
    \label{eq:eigenvalue equation Li-Yau}
    \begin{cases}
    v > 0 \quad \text{ on } M\setminus\overline{B_r} \\
    \Delta v = v.
    \end{cases}
\end{equation}
Then, there exists a positive constant $C = C(n, \gamma, B)>0$ such that
\begin{equation}
    \label{eq:Li-yau subquadratic}
    \frac{|\nabla v(x)|}{\lambda(R)} \leq Cv(x) \quad \forall x \in B_{\gamma R}\setminus \overline{B_R}.
\end{equation}
\end{theorem}

Using these two results we can prove the following.

\begin{theorem}
\label{thm:Lp pp subquadratic}
Let $(M, g)$ be a complete Riemannian manifold satisfying \eqref{eq:ricci lambda} in the sense of quadratic forms. 
Then $M$ is $L^p$-positivity preserving for all $p \in [1, +\infty]$. 
\begin{proof}
Let $u \in L^p(M)$, $p \in [1, +\infty]$ such that $-\Delta u + u \geq 0$ in the sense of distributions. 
Take $\phi \in C^\infty_0(M)$, $\phi \geq 0$ we need to show that
\begin{equation*}
    \int_M u \phi dV_g \geq 0. 
\end{equation*}
To this end we take $v \in C^\infty(M)$, $v>0$ as in \Cref{lem:Lp pp smooth} such that $-\Delta v + v = \phi$ and $v, \Delta v \in L^q(M)$ $\forall q \in [0, +\infty]$. 
Then we proceed as in \Cref{thm:L^p pp Cartan hadamard} using the cut-off functions of \Cref{thm:cutoffs of IRV} instead of the one of \Cref{lem:existence of cutoffs}. 
The proof differs from the one of \Cref{thm:L^p pp Cartan hadamard} only in the estimate of the terms containing $\Delta \chi_R$ and $\nabla \chi_R$. 
The former is immediate: since $|\Delta \chi_R| \leq C_2$ we have
$|uv\Delta \chi_R| \leq C_2 |uv| \in L^1(M)$ hence
\begin{equation*}
    \int_{M} uv\Delta \chi_R dV_g \to 0
\end{equation*}
by dominated convergence as $R \to + \infty$. 
For the latter term, we observe that if $r>0$ is large enough then $\Delta v = v$ on $M \setminus \overline{B_r}$ and $v > 0$ , thus, we have the validity of the Li-Yau estimate of \Cref{thm:Li-Yau of Bianchi and Setti}. 
Since $\nabla \chi_R$ is compactly supported in $B_{\gamma R}\setminus \overline{B_R}$, then 
\begin{equation*}
    |u \langle \nabla \chi_R, \nabla v\rangle| \leq C_2 |u| \frac{|\nabla v|}{\lambda(R)} \leq C C_2 |u| |v| \in L^1(M) \quad \forall x \in M.
\end{equation*}
It follows that 
\begin{equation*}
    \int_{M} u \langle \nabla \chi_R, \nabla v\rangle dV_g \to 0
\end{equation*}
as $R\to + \infty$ which concludes the proof of the theorem. 
\end{proof}
\end{theorem}
\begin{remark}\label{rmk:hsu}
As a consequence of the case $p = +\infty$, we immediately get that the manifold at hand is stochastically complete.
P. Hsu in \cite{Hsu1989} proved the stochastic completeness assuming the $\Ric(x)\ge -\kappa (r(x))$, where $\kappa$ is non decreasing and $\int^\infty \kappa^{-1}=\infty$. 
Keeping also in account that the choice of $\lambda$ in our result can be slightly generalized, \cite[Proposition 1.1]{IRV2020}, our function $\lambda$ is essentially the maximal one admissible in order to fulfill $\int^\infty \lambda^{-1}=\infty$.
\end{remark}
\bibliography{biblio}
\bibliographystyle{acm}
\end{document}